\documentclass[pdflatex,sn-mathphys-num]{sn-jnl}

\newif\ifclean

\cleanfalse

\usepackage{graphicx}%
\usepackage{multirow}%
\usepackage{amsmath,amssymb,amsfonts}%
\usepackage{amsthm}%
\usepackage{mathrsfs}%
\usepackage[title]{appendix}%
\usepackage{xcolor}%
\usepackage{textcomp}%
\usepackage{manyfoot}%
\usepackage{booktabs}%
\usepackage{algorithm}%
\usepackage{algorithmicx}%
\usepackage{algpseudocode}%
\usepackage{listings}%
\usepackage{dsfont}

\usepackage{tcolorbox}
\tcbuselibrary{skins, breakable}
\newtcolorbox[auto counter,number within=section]{approach}[2][]{%
	colback=black!10!white,colframe=white!50!black,arc=0pt,coltitle=black,boxrule=0.6pt,
	breakable,
	fonttitle=\bfseries,
	title=Approach~\thetcbcounter: #2,#1}

\usepackage{tikz-cd}
\usepackage{enumitem}

\algrenewcommand\algorithmicrequire{\textbf{Input:}}
\algrenewcommand\algorithmicensure{\textbf{Output:}}

\newcommand{\nmat}{n_{\textnormal{mat}}}
\newcommand{\nbin}{n_{\textnormal{bin}}}

\newcommand{\nspec}{n_{\textnormal{spec}}}
\newcommand{\emax}{E_{\textnormal{max}}}

\def\R			{\mathbb R}

\DeclareMathOperator{\supp}{supp}

\DeclareMathOperator{\im}{Im}

\newtheorem{definition}{Definition}[section]
\newtheorem{theorem}[definition]{Theorem}
\newtheorem{lemma}[definition]{Lemma}

\usepackage{xcolor}

\makeatletter
\renewcommand*{\@textcolor}[3]{%
	\protect\leavevmode
	\begingroup
	\color#1{#2}#3%
	\endgroup
}
\makeatother

\begin{document}

\title[]{Enhanced attenuation modelling for multispectral computed tomography}

\author*[1]{\fnm{J\"urgen} \sur{Jeschke}}\email{jjeschke@uni-bremen.de}

\author[1]{\fnm{Dirk} \sur{Lorenz}}\email{d.lorenz@uni-bremen.de}

\affil*[1]{\orgdiv{Center for Industrial Mathematics}, \orgname{University of Bremen}, \orgaddress{Postfach 330440, \postcode{28334} \state{Bremen}, \country{Germany}}}

\abstract{We revisit the modelling of the material decomposition problem in multispectral computed tomography and propose to leverage the known absorption spectra of different materials. With this we both derive new uniqueness conditions for the material decomposition in the case that the attenuation can mainly be described by the photoelectric effect and also offer a three-step reconstruction model. Our uniqueness criterion only requires certain characteristics of the bin sensitivity functions and the material attenuation spectra and hence, can be computed a-priori before measurements are taken.}


\maketitle

\section{Introduction}
The idea of multispectral computed tomography (MSCT) was introduced in the 1970s~\cite{alvarez1976energyselective}. This method is also referred to and well described as spectral multi-detector computed tomography (sMDCT)~\cite{Tang2023}. The basic concept is to extend conventional CT as recounted in~\cite{Buzug2008} by using X-ray detectors that measure the intensity in multiple energy bins. Under the assumption that both the emission spectrum of the source and the attenuation spectra of different materials are known, one attempts to reconstruct an artifact free image. If the object is composed of a few known materials, one also aims to compute the distributions of these materials, a problem called material decomposition problem~\cite{le2011least}. This problem can be split into two separate steps: First obtaining the sinograms of the individual materials and then use a standard tomographic inversion method to compute the material distributions from the sinograms. The problem of inverting the sinograms is well understood, as it is used for regular CT where one assumes that the absorption is independent of the used X-ray energy or only a single energy is used~\cite{prohaszka2024derivative}.

\subsection{Multispectral computed tomography model}
\label{subsec:msct-model}
We now recall the standard model for MSCT. Let  $\Omega := \R^2$ be the two-dimensional space and $x \in \Omega$ a point in it. Then for $E \in \R_{\geq0} $ as the energy of the X-ray beam, the (unknown) \emph{attenuation coefficient} is denoted by $\mu(x,E)$ and describes the absorption of said X-ray at the point $x$. The intensity of the X-ray beam after passing through the scanned object is then measured at the detector.
 Each detector pixel measures the incoming energy in $\nbin$ bins according to (known) \emph{sensitivity functions} $s_b(E)$, for $b=1,\dots,\nbin$, 
 i.e. the measurement of a multispectral sensor is split into several bins according to these sensitivity functions. These functions depend on the construction of the detector but also incorporate the emission spectrum of the X-ray source.
 We can assume for both that $\mu(x,E) \geq 0$ and  $s_b(E) \geq 0$. For a line $\ell \subset \Omega $ from the X-ray source to a detector pixel, the reading at the corresponding detector pixel in bin $b$ is hence given by
\begin{align}
	\label{eq:forward-model-emax}
	Y_b(\ell) = \int\limits_0^{\emax} s_b(E)\, \exp\!\Bigl( - \int_\ell \mu(x,E)\, \mathrm{d}\ell(x) \Bigr) \, \mathrm{d}E, \quad b = 1,\dots,\nbin,
\end{align}
where the integral upper bound $\emax$ is the maximum energy value of the source, which is enforced by the operating voltage of the X-ray source and also assumed to be known.

The MSCT problem is to recover $\mu(x,E)$ from measurements $Y_b(\ell)$, and for infinitely many bins it is undetermined. Hence, we focus on the \emph{material decomposition problem}, where it is assumed that only a finite number of $\nmat$ distinct materials are present, each with (known) \emph{attenuation spectra} $\mu_m(E)$ and (unknown) \emph{spatial densities} $f_m(x)$, for $m=1,\dots,\nmat$. With this assumption we can represent the energy and space dependent attenuation coefficient as the sum over the number of the material densities multiplied by the attenuation spectra, i.e.
\begin{align}\label{eq:M-sum}
	\mu(x,E) =  \sum_{m=1}^{\nmat} \mu_m(E) f_m(x).
\end{align}
Inserting \eqref{eq:M-sum} in \eqref{eq:forward-model-emax} we get
\begin{equation*}
	\begin{aligned}
		\label{eq:forward-model2}
		Y_b(\ell) 
		& = \int\limits_0^{\emax} s_b(E) \exp\Bigg( -\sum_{m=1}^{\nmat} \mu_m(E) \int_\ell f_m(x) \, \mathrm{d}\ell(x) \Bigg) \, \mathrm{d}E , \quad b = 1,\dots,\nbin.
	\end{aligned}
\end{equation*}
In the case of a parallel beam architecture, we can characterize each line $\ell$ the X-ray travels by its normal vector $\theta \in [0,\pi)$ and its distance $r \in \R$ from the origin with
\begin{align*}
	\ell_{\theta,r} = \left\{ \left(r \cos ( \theta) - s \sin ( \theta) , r \sin (\theta) + s \cos (\theta) \right) \mid s \in \R \right\}.
\end{align*} 
With this we can identify each sensor position by a point in $S := [0, \pi) \times \R$ and model  the sensor reading as a function $Y : S \to \R^{\nbin}$ with 
\begin{align*}
	Y_b(\theta,r) = Y_b(\ell_{\theta,r}), \quad b = 1,\dots,\nbin.
\end{align*}
For each of those line integrals  we can write 
\begin{align*}
	\int_{\ell_{\theta,r}} f_m(x) \, \mathrm{d}\ell_{\theta,r}(x) = \int_{\R} f\left(r \cos ( \theta) - s \sin ( \theta) , r \sin (\theta) + s \cos (\theta) \right)\mathrm{d}s = \mathcal{R}f_m(\theta,r),
\end{align*}
where $\mathcal{R}$ denotes the Radon transform. This gives us the \emph{MSCT forward model} as
\begin{align}
	\label{eq:MSCT-forward-model}
	Y_b = \int\limits_0^{\emax} s_b(E) \exp\Bigg( -\sum_{m=1}^{\nmat} \mu_m(E) \mathcal{R}f_m \Bigg) \, \mathrm{d}E , \quad b = 1,\dots,\nbin,
\end{align}
see also \cite{prohaszka2024derivative}. 

To obtain the \emph{two step} method, we define the \emph{bin-sinogram function} $\Phi \colon \R^{\nmat} \to \R^{\nbin}_{>0}$ as follows: For a vector $z \in \R^{\nmat}$ with entries $z_{m}$ that represent the total absorption caused by the $m$th material (i.e. an entry in the sinogram of the $m$th material) the measurement of the $b$th bin is 
\begin{equation} \label{eq:Phi}
	\Phi_b(z) = \int\limits_0^{\emax} s_b(E) \exp \left( -\sum\limits_{m=1}^{\nmat} \mu_m(E) z_m \right) \, \mathrm{d}E , \quad b = 1,\dots,\nbin.
\end{equation}
With this we can represent the MSCT forward model as $Y = \Phi ( \mathcal{R}f_1,\dots,\mathcal{R}f_{\nmat}) $. Hence, the two step reconstruction method is: First solve the \emph{bin-sinogram decomposition problem} to get the material sinograms $\mathcal{R}f_m$ from the bin measurements $Y$ and second apply the \emph{Radon transform inversion} to obtain the material densities $f_m$ from $\mathcal{R}f_m$~\cite{Bal_2020}.

Out of~\cite{Bal_2020,bal2022inversion,Gao_2024} it is known that uniqueness and stability estimates for the non-linear bin-sinogram decomposition problem can be obtained by evaluating  the Jacobian. Often this is done by looking at $-\ln(\Phi)$ with the resulting matrix $J(z)$ as
\begin{equation*} \label{eq:DPhi} 
	J_{b,m}(z) =  \left(\Phi_b (z)\right)^{-1} \int\limits_0^{\emax} s_b(E) \, \mu_m(E) \, \exp \left( -\sum\limits_{m'=1}^{\nmat} \mu_{m'}(E) z_{m'} \right) \, \mathrm{d}E,
\end{equation*}
for $b = 1, \dots, \nbin$ and $m = 1, \dots, \nmat$.
 Hadamard global inverse function theorem~\cite{Ruzhansky2014inversion} states then that a differentiable mapping from $ \R^{\nmat} \to \R^{\nbin} = \R^{\nmat}$ is a diffeomorphism and therefore well-posed if the mapping is proper and the Jacobian never vanishes.
 When restricting the domain of the mapping to a closed rectangular region $\mathcal{R} \subset \R^{\nmat}$, theorem 3 out of~\cite{Bal_2020} gives injectivity in the case that $\nmat = \nbin = 2$ if Jacobian never vanishes on $\mathcal{R}$.
 For $\nmat = \nbin \geq 3$, this condition on $J(z)$ is not sufficient. Using theorem 4 out of~\cite{Gale1965} one obtains injectivity on $\mathcal{R}$ if $J(z)$ is a  $P$-matrix for every $z \in \mathcal{R}$, i.e. that it has only positive principal minors~\cite{Fiedler1962}. This also implies that the Jacobian is non vanishing, but by itself is a much stricter condition.
 From this one can also derive stability results~\cite{Bal_2020}.
 In practice all these approaches using the Jacobian without any further assumptions on $\Phi$ are usually difficult to verify, as one has to do calculations for every point in $\mathcal{R}$. 
 
 Other methods focus on the material attenuation spectra~\cite{Tang2021} or the sensitivity functions~\cite{Ren2021} to identify well-posedness properties of the individual parts, however omitting the nonlinearity of the entire problem.

\subsection{Contributions}

In this paper we propose a further division of the material decomposition problem by using the main physical aspects of the interaction of X-ray photons with matter, which are shared between all materials. This allows us to model the absorption of the individual materials based on a set of modelled X-ray interactions.
The influence of the individual materials can then be reconstructed by solving a linear problem.
Effectively this decomposes the problem into three steps and this allows for a deeper understanding of the problem as only one of these steps is non-linear, and the other steps are the well understood Radon inversion and the finite dimensional problem of material-energy decomposition which can be done separately for every detector pixel. 
 For the X-Ray interactions we will limit us to the photoelectric effect, however we will also describe how Compton scattering can be incorporated.
We demonstrate that this three step approach for the material decomposition can be applied in practice, and also that it allows to derive a simple criterion for uniqueness of the reconstruction problem. 

The rest of the paper is organized as follows: In Section~\ref{sec:attenuation-model} we refine the model for the absorption spectra and propose the three step model that uses this refined model. In Section~\ref{sec:bin-sino-decomp} we analyze the bin-sinogram decomposition problem and in Section~\ref{sec:Num_Ex} illustrate the resulting three step reconstruction method. Finally, Section~\ref{sec:conclusion} draws some conclusion.

\section{Modelling of attenuation spectra and three step model}
\label{sec:attenuation-model}
To better understand the bin-sinogram decomposition problem we take a closer look at the properties of the known attenuation spectra by taking the physical laws governing the interaction of photons and matter into account. The main effect to consider for low energy X-ray photons is the photoelectric effect where a photon is absorbed with subsequent ejection of an electron. Due to the quantified nature of the interaction it is required that the photon energy is greater than the binding energy of the electron~\cite{UBHD-67757941}. For electrons in the K shell one can describe the absorption chance by an inverse power law dependent on the photon energy. The L and M shells differ only by an additional factor~\cite{1191161}. We denote by $E_m^{(K)}$, $E_m^{(L)}$, and $E_m^{(M)}$ the binding energies of the K, L, and M shell of the $m$-th material respectively, and by $a_m^{(K)}$, $a_m^{(L)}$ and $a_m^{(M)}$ the corresponding factors for the absorption chance. Then the attenuation spectrum can be written as
\begin{align}
	\label{eq:material-attenuation-spectrum-shells}
	\mu_m(E) = E^{-c} \left(a_m^{(K)} \cdot \mathds{1}_{\geq E_m^{(K)}}(E) +  a_m^{(L)} \cdot \mathds{1}_{\geq E_m^{(L)}}(E) + a_m^{(M)} \cdot \mathds{1}_{\geq E_m^{(M)}}(E) \right) 
\end{align}
for $m = 1,\dots,\nmat$, where the exponent $c>0$ is chosen to best fit the real absorption function. In~\cite{1191161} a value of $c = 3.5$ is suggested and  \cite{4121176,2017SenIm..18....5R,Bunker_2010} use $c = 3$. 
In figure~\ref{fig:REC} a comparison of the attenuation spectrum of iodine, gadolinium, and water with the representation using \eqref{eq:material-attenuation-spectrum-shells} is shown for different values of $c$.
As this representation of the attenuation coefficient only uses the photoelectric effect and omits the contribution of Compton scattering, it is most of all applicable for lower photon energies or materials with larger atomic numbers~\cite{alvarez1976energyselective,4121176}. This behavior can also be seen in figure~\ref{fig:REC} for water.  If one wants to additionally incorporate Compton scattering, it can be done by adding an another constant  to \eqref{eq:material-attenuation-spectrum-shells}, as the  Klein–Nishina function describing the effect only varies a little with photon energy in the applicable  X-ray range~\cite{alvarez1976energyselective,2017SenIm..18....5R}. The attenuation spectra would then have to be represented as $\mu_m^{\text{Total}}(E) = \mu_m(E) + \mu_m^{\text{Compton}}$. For the further study we will assume that $\mu_m(E) \gg \mu_m^{\text{Compton}}$ and we can therefore approximate $\mu_m^{\text{Total}}(E) \approx \mu_m(E)$.
\begin{figure}[t]
		\begin{enumerate}[label=(\alph*)]
			\item\label{c=2.6Rec}
				\includegraphics[width=0.925\textwidth]{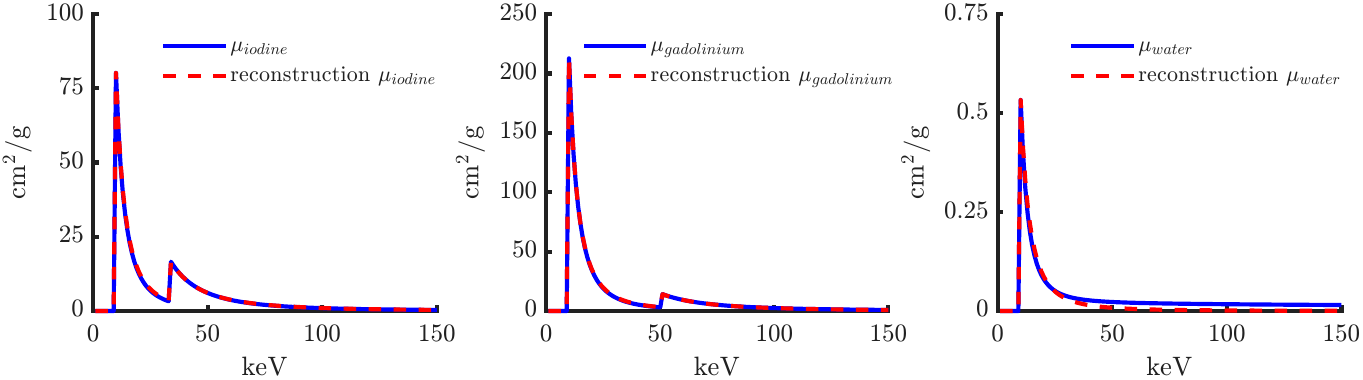}
				\bigskip
			\item\label{c=3.0Rec}
				\includegraphics[width=0.925\textwidth]{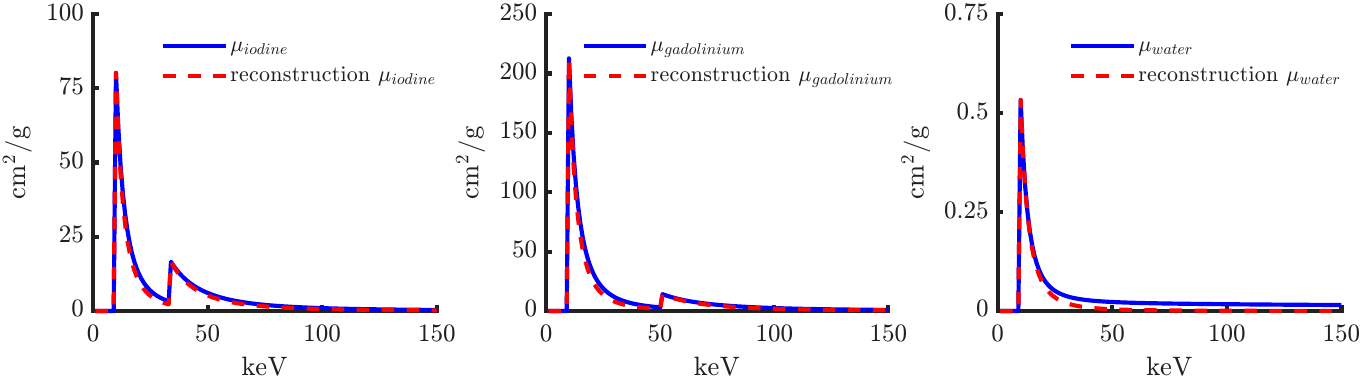}
				\bigskip
			\item\label{c=3.5Rec}
				\includegraphics[width=0.925\textwidth]{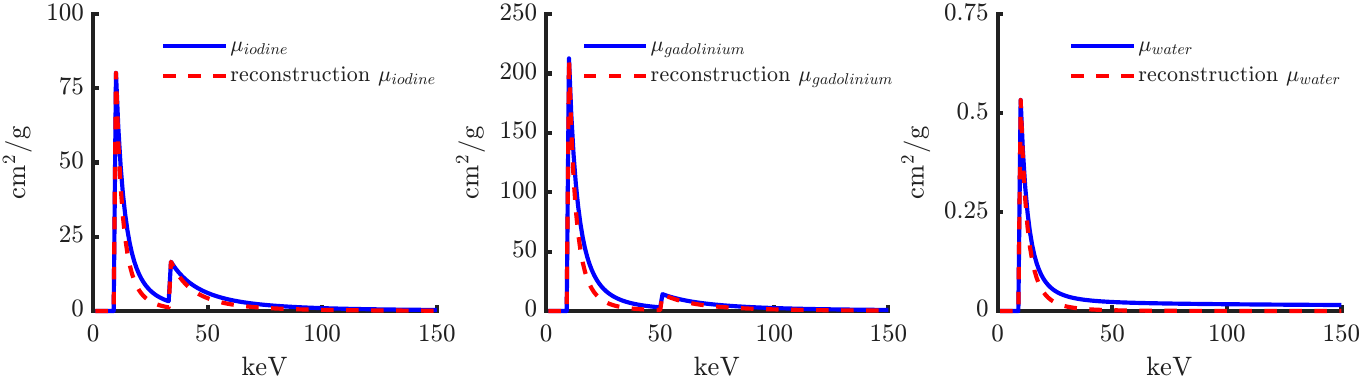}
		\end{enumerate}
		\caption{Mass attenuation coefficient of iodine, gadolinium and water out of~\cite{Mory_2018,Mory_2019_Git} and the representation using \eqref{eq:material-attenuation-spectrum-shells} with \ref{c=2.6Rec} $c = 2.6 $, \ref{c=3.0Rec} $c = 3 $ and \ref{c=3.5Rec} $c = 3.5 $.}
		\label{fig:REC}
\end{figure}

Let $E_{\textnormal{shells}} = \{ E_1^{(K)}, E_1^{(L)},E_1^{(M)}, \dots, E_{\nmat}^{(M)} \}$ be the set of binding energies, $\nspec =|\{ E \in E_{\textnormal{shells}} : E < \emax \}| $ the number of binding energy values smaller than $\emax$ and $E_{\textnormal{shells}}^{\emax} = \{ E_1,\dots,E_{\nspec},\emax  \} \subset E_{\textnormal{shells}} \cup \{\emax \}$ the set of the smallest $\nspec$ binding energies and the maximum X-ray energy. For $E_{\textnormal{shells}}^{\emax}$ we assume that $E_1 < E_2< \dots <E_{\nspec}<\emax$. Also let the \emph{spectral intervals} $I_1,\dots,I_{\nspec}$ be so that $ I_1 = [E_1,E_2)$, and so on until $I_{\nspec} = [E_{\nspec},\emax )$. With this intervals we can rewrite \eqref{eq:material-attenuation-spectrum-shells} as
\begin{align}
\label{eq:material-attenuation-spectrum}
\mu_m(E) =E^{-c} \left( \sum_{i=1}^{\nspec}  a_m^{(i)} \cdot \mathds{1}_{I_i} (E) \right),
\end{align}
where $a_m^{(i)}  \in \{ a_m^{(K)} , a_m^{(K)} + a_m^{(L)}, a_m^{(K)} + a_m^{(L)} + a_m^{(M)} \}$ for $i = 1,\dots,\nspec$.

Applying the modelling of the material attenuation spectra to the MSCT model by replacing the material attenuation function with its representation \eqref{eq:material-attenuation-spectrum} the  MSCT forward model \eqref{eq:MSCT-forward-model} becomes
\begin{equation}
\label{eq:new-MSCT-forward-model}
\begin{aligned}
	Y_b &= \int\limits_0^{\emax} s_b(E) \exp\Bigg( -\sum_{m=1}^{\nmat} \mu_m(E) \mathcal{R}f_m \Bigg) \, \mathrm{d}E \\
	&= \int\limits_0^{\emax} s_b(E) \exp\Bigg( -  \sum_{m=1}^{\nmat}  \cdot \left( E^{-c} \cdot \sum_{i=1}^{\nspec}  a_m^{(i)} \cdot \mathds{1}_{I_i} (E) \right) \mathcal{R}f_m \Bigg) \, \mathrm{d}E \\
	&=  \sum_{j=1}^{\nspec} \int\limits_{I_j}  s_b(E) \exp\Bigg( -  \sum_{m=1}^{\nmat}  \cdot \left( E^{-c} \cdot \sum_{i=1}^{\nspec}  a_m^{(i)} \cdot \mathds{1}_{I_i} (E) \right) \mathcal{R}f_m \Bigg) \, \mathrm{d}E \\
	&= \sum_{i=1}^{\nspec} \int\limits_{I_i} s_b(E) \exp\Bigg( - E^{-c} \cdot \sum_{m=1}^{\nmat} a_m^{(i)} \cdot  \mathcal{R}f_m \Bigg) \, \mathrm{d}E
\end{aligned}
\end{equation}
for $ b = 1,\dots,\nbin$, and the bin-sinogram function \eqref{eq:Phi} changes to
\begin{equation}
\label{eq:new-Phi}
\Phi_b(z) = \sum_{i=1}^{\nspec} \int\limits_{I_i} s_b(E) \exp\Bigg( - E^{-c} \cdot \sum_{m=1}^{\nmat} a_m^{(i)} \cdot z_m \Bigg) \, \mathrm{d}E
\end{equation}
for $z \in \R^{\nmat}$ and $ b = 1,\dots,\nbin$.

Using this we can split the bin-sinogram function into a linear part represented by the matrix $A \in \R^{\nspec \times \nmat}$ with
\begin{align}
\label{eq:A}
A = \begin{pmatrix}
	a_1^{1} & \cdots & a_{\nmat}^{1} \\ \vdots & & \vdots \\ a_1^{\nspec} & \cdots & a_{\nmat}^{\nspec}
\end{pmatrix}
\end{align}
and nonlinear part $\Psi: \R^{\nspec} \to \R^{\nbin}_{>0} $ given by
\begin{equation} \label{eq:Psi}
\Psi_b(v) =\sum_{i=1}^{\nspec} \int\limits_{I_i} s_b(E) \exp\Big( -  E^{-c} \cdot v_i \Big) \, \mathrm{d}E
\end{equation}
for $v \in \R^{\nspec}$ and $b \in  \{1,\dots,\nbin\}$. 
In the case of not negligible Compton scattering $A$ would be a matrix of the size $\nspec +1 \times \nmat$, and $\Psi$ would depend on $\nspec+1$ variables, where the additional variable is given by the sum over the Compton parts of all materials. The exponential function of the negative value of this extra variable would then be multiplied with the whole sum.
Regardless of this the bin-sinogram function then becomes
\begin{equation}\label{eq:Phi_Psi_A}
\Phi (z) = \Psi ( A \cdot z)
\end{equation}
for $z \in \R^{\nmat}$ and the MSCT forward model can now be divided into three parts with 
\begin{equation}\label{eq:Y_Phi_Psi_A}
Y = \Psi ( A \cdot ( \mathcal{R}f_1,\dots,\mathcal{R}f_{\nmat})^T ).
\end{equation}
\begin{figure}
\centering   	
\begin{tikzcd}
	\text{bin-sinogram} &  Y:S \to \R^{\nbin}	\arrow[d, shift left=15, bend left=50, "\textnormal{bin-decomposition}"]\\
	\color{cyan} \text{spectral-sinogram} 	 \arrow[u, shift left=15, bend left=50, cyan, "\Psi"]  & \color{cyan} v : S \to \R^{\nspec} \arrow[d, shift left=15, bend left=50, "\textnormal{material-energy decomposition}"]\\
	\text{material-sinogram}   \arrow[u, shift left=15, bend left=50, cyan,"A"] \arrow[uu, shift left=20, bend left=50, "\Phi"]  & z: S \to \R^{\nmat} \\
	\text{material-density}   \arrow[u, shift left=15, bend left=50, "\mathcal{R}"]  & f: \Omega \to \R^{\nmat}	
\end{tikzcd}
\caption{Depiction of the two step method consisting of the Radon transform $\mathcal{R}$ followed by the bin-sinogram function $\Phi$ and the three step representation where the bin-sinogram function is split into the material-energy mapping $A$ and the bin reading function $\Psi$ as highlighted in the depiction.}
\label{fig:ThreeStep}
\end{figure}
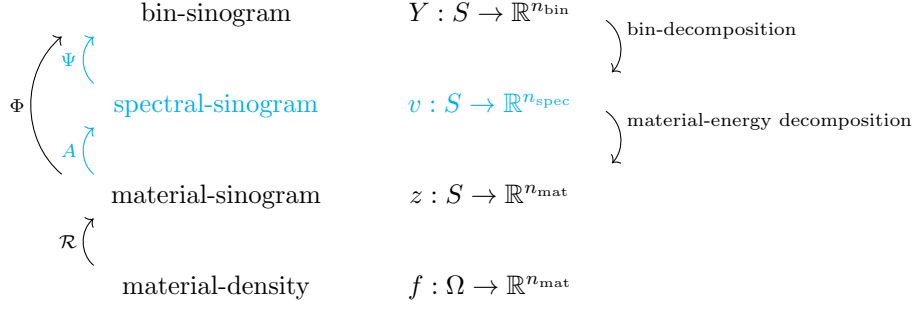

The matrix $A$ gives a representation of the X-ray absorption properties at distinct energy values, where
\begin{align*}
E^{c} \cdot \int\limits_\ell \mu(x,E)\, \mathrm{d}\ell(x) = \left( A \cdot \begin{pmatrix}
	\int_\ell f_1(x) \, \mathrm{d}\ell(x) \\ \vdots \\ \int_\ell f_{\nmat}(x) \, \mathrm{d}\ell(x) 
\end{pmatrix} \right)_i
\end{align*}
for $E \in I_i$, and then $\Psi$ turns these into the bin readings. Hence, we call $A$ the \emph{material-energy mapping} and $\Psi$ the \emph{bin-reading function}. The bin-sinogram decomposition problem can now also be split into the linear \emph{material-energy decomposition problem} of finding the material-sinogram given spectral-sinogram and the \emph{bin-decomposition problem} of getting the spectral-sinogram from the bin-sinogram, see figure~\ref{fig:ThreeStep}.

Inverting each of those steps we derive a three step reconstruction approach for the MSCT problem as follows:
\begin{approach}[label={Approche1},nameref={Three step approach with sinogram based material decomposition}]{Three step approach with sinogram based material decomposition}
\begin{enumerate}
	\item \emph{Bin-decomposition problem}: Find $v : S \to \R^{\nspec}$ 
	\begin{align*}
		\Psi(v(\theta,r)) = Y(\theta,r)
	\end{align*}
	for each point $(\theta,r)$
	\item \emph{material-energy decomposition problem}: Find $z: S \to \R^{\nmat}$ so that
	\begin{align*}
		Az(\theta,r) = v(\theta,r)
	\end{align*}
	for each point $(\theta,r)$
	\item \emph{Radon transform inversion}: Find $f : \Omega \to \R^{\nmat}$ so that
	\begin{align*}
		\mathcal{R}f_m = z_m
	\end{align*}
	for all $m = 1,\dots,\nmat$.
\end{enumerate}
\end{approach}

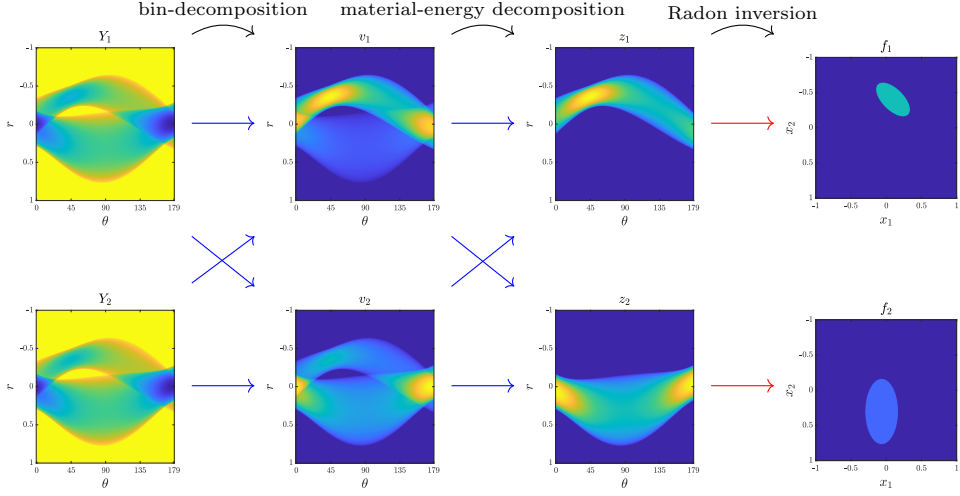
\begin{figure}
	\centering   	
	\begin{tikzcd}
		\includegraphics[width=0.175\linewidth]{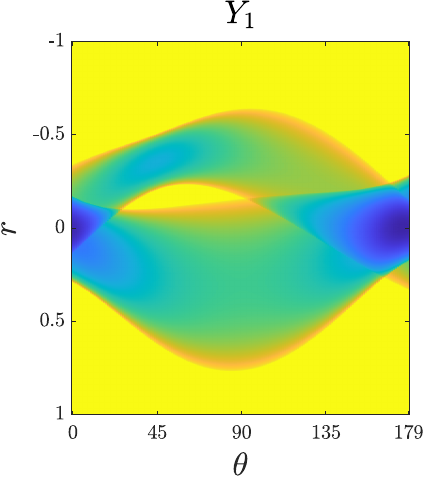} \arrow[r, shift left=20, bend left, "\textnormal{bin-decomposition}"] \arrow[r, shift left=15, blue] \arrow[dr, start anchor=south east, end anchor=north west, blue]&
		\includegraphics[width=0.175\linewidth]{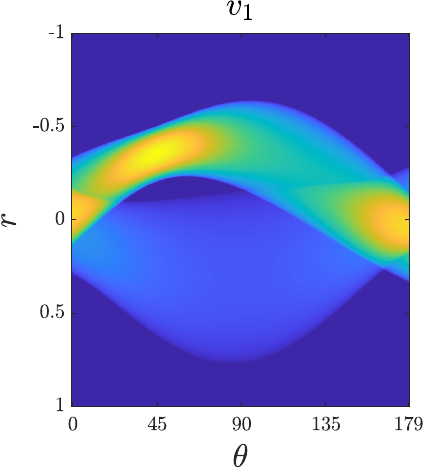} \arrow[r, shift left=20, bend left, "\textnormal{material-energy decomposition}"] \arrow[r, shift left=15, blue] \arrow[dr, start anchor=south east, end anchor=north west, blue] &
		\includegraphics[width=0.175\linewidth]{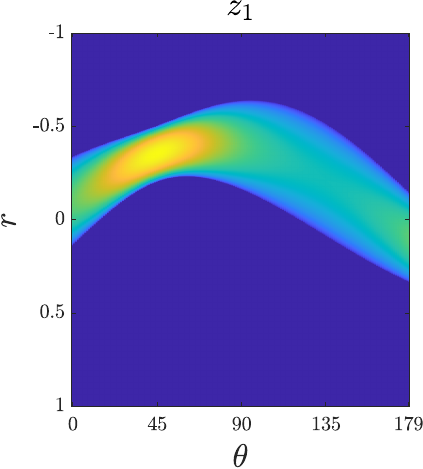} \arrow[r, shift left=20, bend left, "\textnormal{Radon inversion}"] \arrow[r, shift left=15, red] &
		\includegraphics[width=0.175\linewidth]{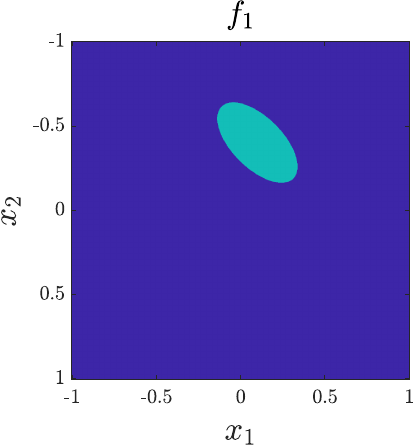} \\
		\includegraphics[width=0.175\linewidth]{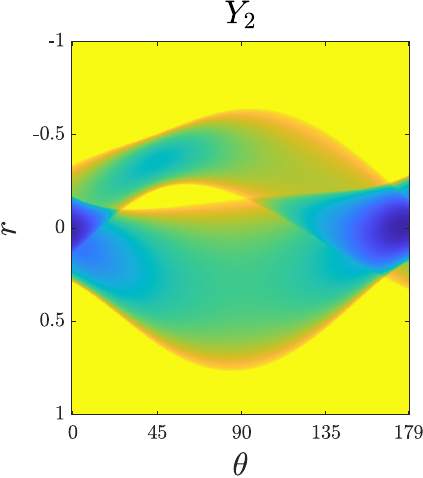}  \arrow[r, shift left=15, blue] \arrow[ur, start anchor=north east, end anchor=south west, blue] &
		\includegraphics[width=0.175\linewidth]{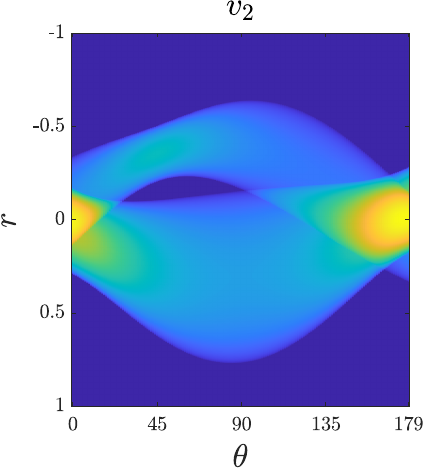} \arrow[r, shift left=15, blue] \arrow[ur, start anchor=north east, end anchor=south west, blue] &
		\includegraphics[width=0.175\linewidth]{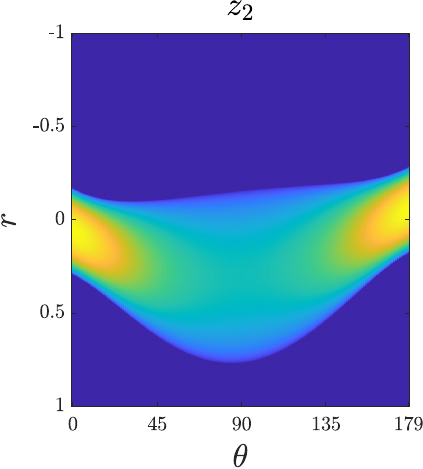} \arrow[r, shift left=15, red] & \includegraphics[width=0.175\linewidth]{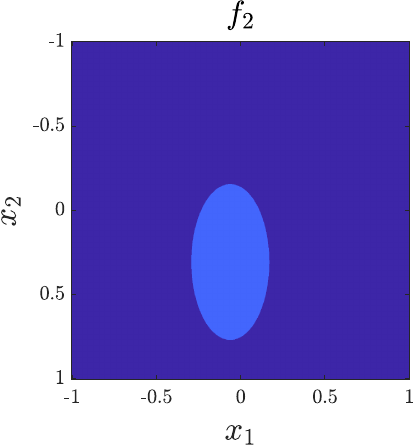}
	\end{tikzcd}
	\caption{Depiction of the reconstruction steps of approach~\ref{Approche1}, where a blue arrow denotes a dependency from a single pixel in the sinograms and a red  denotes a dependency of the entire function. In the first step we reconstructed for each energy spectral-interval $I_1$ and $I_2$ the spectral-sinograms $v_1$ and $v_2$ for each pair of $(\theta,r)$ individually based on the corresponding points in bin-sinograms  $Y_1$ and $Y_2$. With the same pointwise dependency we obtain for each material the material-sinograms $z_1$ and $z_2$ from $v_1$ and $v_2$. For the third step we calculate the whole material-densities $f_1$ with all points of $z_1$ and $f_2$ with $z_2$.}
	\label{fig:ThreeStep_1}
\end{figure} 

The results of the individual steps of approach~\ref{Approche1} are shown in figure~\ref{fig:ThreeStep_1}.
As the solution $v_{\ell} \in \R^{\nspec}$ of the bin-decomposition problem represents the absorption properties at distinct energy levels, we can also apply the Radon inversion immediately to the bin-decomposition and obtain the absorption properties at every point $x \in \R^2$ for the energy values $E_1,\dots,E_{\nspec}$. With the \emph{spectral absorption functions} $f_{E_i} : \Omega \to \R$ for $E_i = E_1,\dots,E_{\nspec}$ given by $f_{E_i} = \sum_{m=1}^{\nmat} a_m^{(i)} \cdot  f_m $ we get
\begin{equation}
\label{eq:new-MSCT-forward-model_2}
\begin{aligned}
	Y_b &= \sum_{i=1}^{\nspec} \int\limits_{I_i} s_b(E) \exp\Bigg( - E^{-c} \cdot \sum_{m=1}^{\nmat} a_m^{(i)} \cdot  \mathcal{R}f_m \Bigg) \, \mathrm{d}E \\
	&= \sum_{i=1}^{\nspec} \int\limits_{I_i} s_b(E) \exp\Bigg( - E^{-c} \cdot \mathcal{R} \Bigg( \sum_{m=1}^{\nmat} a_m^{(i)} \cdot  f_m \Bigg) \Bigg) \, \mathrm{d}E \\
	&= \sum_{i=1}^{\nspec} \int\limits_{I_i} s_b(E) \exp\Bigg( - E^{-c} \cdot \mathcal{R}f_{E_i} \Bigg) \, \mathrm{d}E. \\
\end{aligned}
\end{equation}

The MSCT forward model can then be depicted as 
\begin{equation}\label{eq:Y_Psi_R_A}
Y = \Psi ( \mathcal{R} (A \cdot ( f_1,\dots,f_{\nmat})^T) )
\end{equation}
and we obtain the three step approach with image space based material-energy decomposition. 
\begin{approach}[label={Approche2},nameref={Three step approach with image space based material decomposition}]{Three step approach with image space based material decomposition}
\begin{enumerate}
	\item \emph{Bin-decomposition problem}: Find $v : S \to \R^{\nspec}$ 
	\begin{align*}
		\Psi(v(\theta,r)) = Y(\theta,r)
	\end{align*}
	for each point $(\theta,r)$
	\item \emph{Radon transform inversion}: Find $f_{E_1},\dots,f_{E_{\nspec}} : \Omega \to \R$ so that
	\begin{align*}
		\mathcal{R}f_{E_1}(\theta,r) , \dots ,\mathcal{R}f_{E_{\nspec}}(\theta,r)  = v(\theta,r)
	\end{align*}
	for each point $(\theta,r)$
	\item \emph{material-energy decomposition problem}: Find $f : \Omega \to \R^{\nmat}$ so that 
	\begin{align*}
		A f  = (f_{E_1} , \dots ,f_{E_{\nspec}})^T.
	\end{align*}
\end{enumerate}
\end{approach}

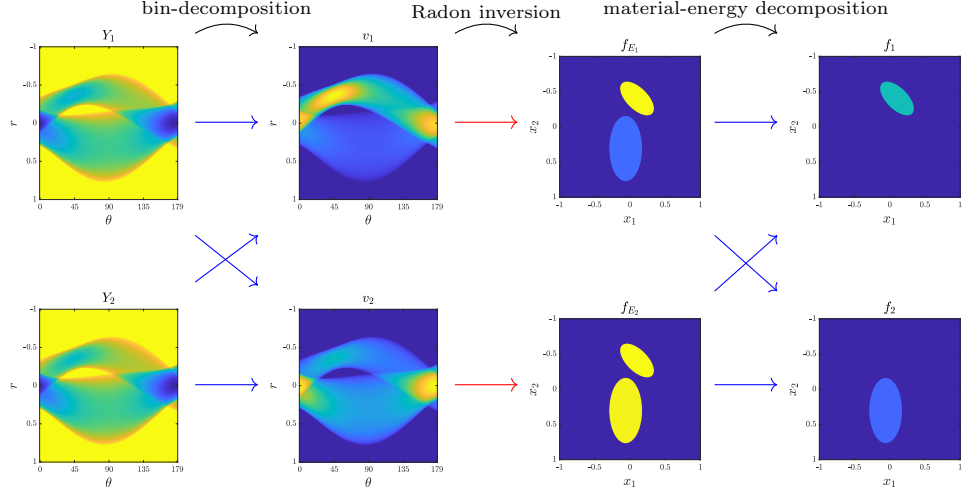
\begin{figure}
	\centering   	
	\begin{tikzcd}
		\includegraphics[width=0.175\linewidth]{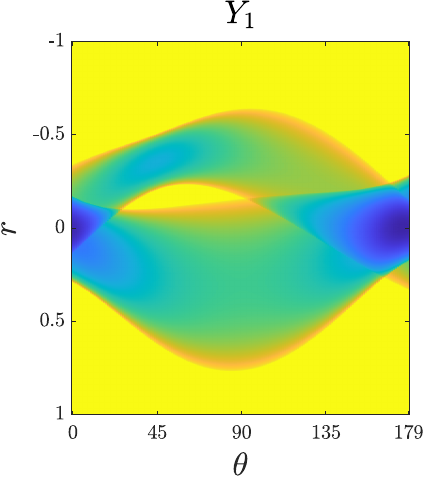} \arrow[r, shift left=20, bend left, "\textnormal{bin-decomposition}"] \arrow[r, shift left=15, blue] \arrow[dr, start anchor=south east, end anchor=north west, blue]&
		\includegraphics[width=0.175\linewidth]{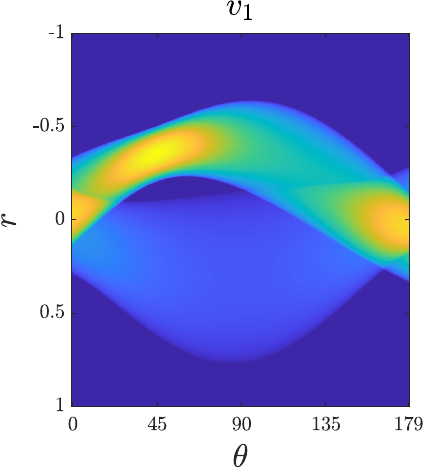} \arrow[r, shift left=20, bend left, "\textnormal{Radon inversion}"] \arrow[r, shift left=15, red] &
		\includegraphics[width=0.175\linewidth]{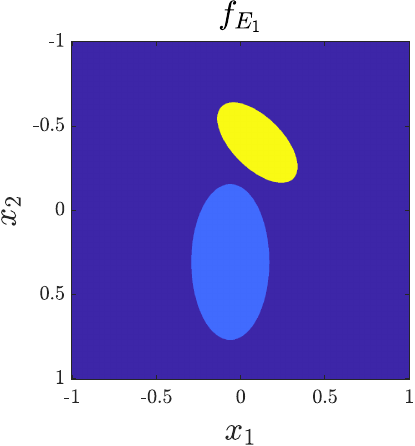} \arrow[r, shift left=20, bend left, "\textnormal{material-energy decomposition}"] \arrow[r, shift left=15, blue] \arrow[dr, start anchor=south east, end anchor=north west, blue] &
		\includegraphics[width=0.175\linewidth]{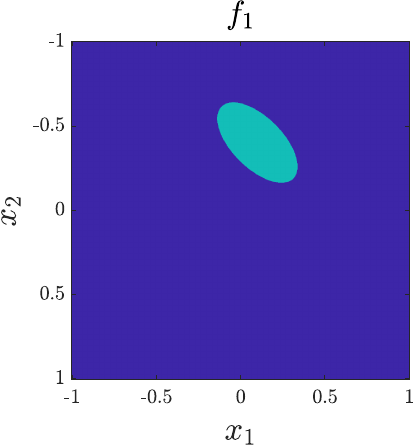} \\
		\includegraphics[width=0.175\linewidth]{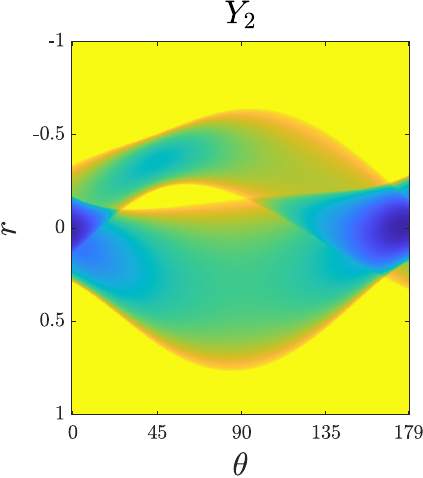}  \arrow[r, shift left=15, blue] \arrow[ur, start anchor=north east, end anchor=south west, blue] &
		\includegraphics[width=0.175\linewidth]{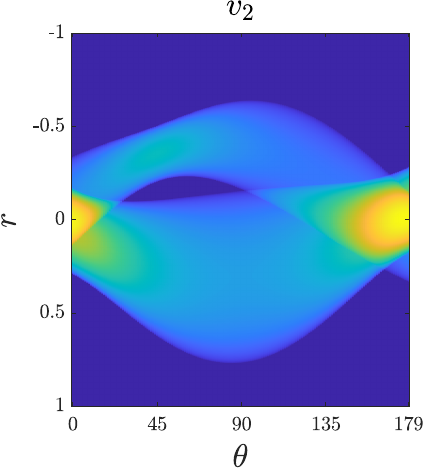} \arrow[r, shift left=15, red] &
		\includegraphics[width=0.175\linewidth]{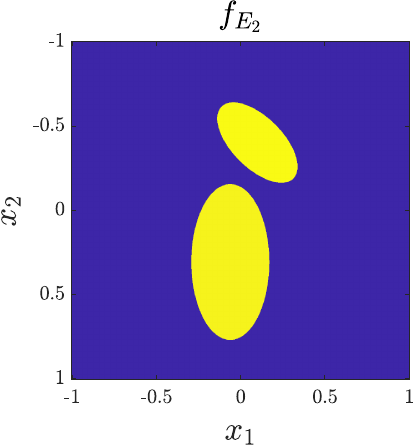} \arrow[r, shift left=15, blue] \arrow[ur, start anchor=north east, end anchor=south west, blue] & \includegraphics[width=0.175\linewidth]{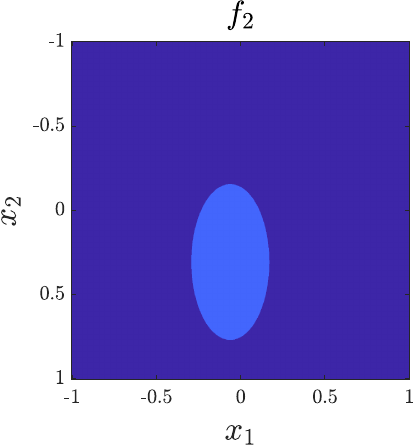}
	\end{tikzcd}
	\caption{Depiction of the reconstruction steps of approach~\ref{Approche2}, where a blue arrow denotes a dependency from a single point and a red  denotes a dependency of the entire function. In the first step we reconstructed for each energy spectral-interval $I_1$ and $I_2$ the spectral-sinograms $v_1$ and $v_2$ for each pair of $(\theta,r)$ individually based on the corresponding points in bin-sinograms $Y_1$ and $Y_2$, then we obtain spectral absorption function $f_{E_1}$ from all of $v_1$ and $f_{E_2}$ from  $v_2$. In the last step we can calculate the material-densities $f_1$ and $f_2$ from  $f_{E_1}$ and  $f_{E_2}$ for each point $(x_1,x_2)$ individually.}
	\label{fig:ThreeStep_2}
\end{figure} 

The results of the individual steps of approach~\ref{Approche2} are shown in figure~\ref{fig:ThreeStep_2}.
The spectral absorption functions $f_{E_1},\dots,f_{E_{\nspec}}$ we obtained in the second step give us the ability to reconstruct $\mu$ for energy values between $E_1$ and $\emax$ as
\begin{align*}
\mu(x,E) =  E^{-c} \cdot \left(  \sum_{i=1}^{\nspec}  \mathds{1}_{I_i} (E) \cdot  f_{E_i}(x) \right).
\end{align*}

\section{Bin-sinogram decomposition problem}
\label{sec:bin-sino-decomp}
The modelling of the attenuation spectra and specifically \eqref{eq:Phi_Psi_A} allows us to investigate the bin-sinogram decomposition problem and the bin-sinogram function $\Phi$ describing the problem in two distinct parts. As the material-energy decomposition is a linear problem, we mainly have to understand the non-linear bin-reading function $\Psi$ and the respective bin-decomposition problem.

\subsection{Bin-decomposition problem}
To investigate the underlying properties of the bin-decomposition problem we will assume that $\nbin = \nspec$. In the simplest case, where each bin is only dependent on a single energy interval $I_i$ we have the following result.
\begin{theorem}\label{theo:diffeomorph_ideal}
Let $\nbin = \nspec$ and the sensitivity functions satisfy $\int_{I_i}s_b(E)  \mathrm{d}E \neq 0$ if and only if $ b = i$, then $\Psi : \R^{\nspec} \to \R_{>0}^{\nbin}$ is a $\mathcal{C}^1$-diffeomorphism.
\end{theorem}
\begin{proof}
To prove this we can use the Hadamard global inverse function theorem~\cite{Ruzhansky2014inversion}. For this we consider $\ln ( \Psi)$ and have to prove that it is a $C^1$-map with non vanishing Jacobian and proper. As a combination and summation of differential functions $\ln ( \Psi)$ is a $C^1$-map. Then for a point $v \in \R^{\nspec}$ we get, because of the $ s_b(E) \geq0$ and following from the assumption that  $\int_{I_i}s_b(E)  \mathrm{d}E \neq 0$ if and only if $ b = i$, that $\Psi_b(v) = \int\limits_{I_b} s_b(E) \exp\Big( -  E^{-c} \cdot v_i \Big)  \mathrm{d}E $ and for the Jacobian at $v$
\begin{align*}
	J(\ln \Psi)_{b,i}  = \frac{\partial}{\partial v_i} [ \ln \Psi_b(v)]&= -\left(\Psi_b(v)\right)^{-1} \int\limits_{I_i} s_b(E) E^{-c} \exp\Big( -  E^{-c} \cdot v_i \Big)  \mathrm{d}E\\
	& \leq -  E_{i+1}^{-c} \frac{ \int\limits_{I_i} s_b(E) \exp\Big( -  E^{-c} \cdot v_i \Big)  \mathrm{d}E }{\int\limits_{I_b} s_b(E) \exp\Big( -  E^{-c} \cdot v_i \Big)  \mathrm{d}E}
\end{align*}
for $b = 1, \dots, \nbin$ and $i = 1, \dots, \nspec$. This gives us that $J(\ln \Psi)_{b,i} \neq 0 $ if and only if $ b = i$  and  $J(\ln \Psi)_{b,b} < 0 $. Since $J(\ln \Psi)$ is a diagonal matrix we have $|\det J(\ln \Psi)| \geq \prod_{i=1}^{\nspec} E_{i+1}^{-c}$. This holds for all $v \in \R^{\nspec}$. To show that $\ln ( \Psi)$ is proper we look at $|v| \to \infty$. In this case we must also have a $i = 1,\dots,\nspec$ with $|v_i| \to \infty$. For $\ln (\Psi_i)$ we then get
\begin{align*}
	\lim\limits_{ |v_i| \to \infty} \left| \ln (\Psi_i(v))  \right| &=	\lim\limits_{ |v_i| \to \infty} \left| \ln \left( \int\limits_{I_i} s_i(E) \exp\Big( -  E^{-c} \cdot v_i \Big) \, \mathrm{d}E \right)\right| \\
	&\geq \lim\limits_{ |v_i| \to \infty} \left| \ln \left( \exp\Big( -  E_{i+1}^{-c} \cdot v_i \Big)  \int\limits_{I_i} s_i(E) \mathrm{d}E \right) \right| \\
	&\geq \lim\limits_{ |v_i| \to \infty} \left| E_{i+1}^{-c} v_i \right|- \left| \ln \left( \int\limits_{I_i} s_i(E) \mathrm{d}E \right) \right| \\
	&= \infty
\end{align*}
and therefore $|\ln (\Psi) | \to \infty$ whenever $|v| \to \infty$. Theorem 2.2 from~\cite{Ruzhansky2014inversion} then gives us that $\ln(\Psi): \R^{\nspec} \to \R^{\nbin}$ is a $\mathcal{C}^1$-diffeomorphism and therefore also that  $\Psi : \R^{\nspec} \to \R_{>0}^{\nbin}$ is a $\mathcal{C}^1$-diffeomorphism.
\end{proof}
Such demands on the sensitivity functions can be given under the assumption of ideal detector responses. In the more realistic case the sensitivity functions overlap~\cite{Tang2022}. For this case we assume that the sensitivity functions follow the scheme that $ \supp (s_b) = [E, \emax ]$ for some energy value $E$ dependent on the specific bin. This gives us the following statement regarding uniqueness.
\begin{theorem}\label{theo:uniqueness_ideal}
Let $\nbin = \nspec$. If the sensitivity functions satisfy that $\int_{I_i}s_b(E)  \mathrm{d}E \neq 0$ if and only if $ b \leq i$, then $\Psi $ is injective.
\end{theorem}
\begin{proof}
Assuming that we have $v,\tilde{v} \in \R^{\nspec}$ with $\Psi(v) = \Psi( \tilde{v})$ we denote by $j$ the largest index from $\{ 1,\dots,\nspec\}$ such that $v_j \neq \tilde{v}_j$. Then we have $v_l = \tilde{v}_l$ for all $j < l \leq \nspec$ and for this $j$ it holds that
\begin{align*}
	&\quad \Psi_j(v) -\Psi_j(\tilde{v}) \\
	&= \sum_{i=1}^{\nspec} \int\limits_{I_i} s_j(E) \exp\Big( -  E^{-c} \cdot v_i \Big) \, \mathrm{d}E - \sum_{i=1}^{\nspec} \int\limits_{I_i} s_j(E) \exp\Big( -  E^{-c} \cdot \tilde{v}_i \Big) \, \mathrm{d}E \\
	&=  \sum_{i=j}^{\nspec} \int\limits_{I_i} s_j(E) \exp\Big( -  E^{-c} \cdot v_i \Big) \, \mathrm{d}E - \sum_{i=j}^{\nspec} \int\limits_{I_i} s_j(E) \exp\Big( -  E^{-c} \cdot \tilde{v}_i \Big) \, \mathrm{d}E \\
	&= \int\limits_{I_j} s_j(E) \exp\Big( -  E^{-c} \cdot v_j \Big) \, \mathrm{d}E -  \int\limits_{I_j} s_j(E) \exp\Big( -  E^{-c} \cdot \tilde{v}_j \Big) \, \mathrm{d}E \\
	&=  \int\limits_{I_j} s_j(E) \exp\Big( -  E^{-c} \cdot v_j \Big )\left( 1 - \exp\Big( -  E^{-c} \cdot (\tilde{v}_j -v_j) \Big )\right) \, \mathrm{d}E.
\end{align*}
Since $v_j \neq \tilde{v}_j$ we have $\exp( -  E^{-c} \cdot (\tilde{v}_j -v_j)) \neq 1$, $\Psi_j(v) -\Psi_j(\tilde{v}) \neq 0$ and consequently a contradiction to the assumption. This gives us that for all $y \in  \im(\Psi)$ there exists  only one $v \in \R^{\nspec}$ so that $\Psi (v) = y$.
\end{proof}

While theorem~\ref{theo:uniqueness_ideal} gives conditions under which the bin-decomposition problem has at most one solution, the problem is still ill-posed, as $\Psi : \R^{\nspec} \to \R_{>0}^{\nbin}$ is not surjective in general. This is shown in the following lemma.
\begin{lemma}\label{lem:surjectivity}
Let $\nbin = \nspec$ and the sensitivity functions satisfy that $\int_{I_i}s_b(E)  \mathrm{d}E \neq 0$ if and only if $ b \leq i$, then for every $\tilde{y}_{\nbin} \in \R_{>0}$ there exist $\tilde{y_1} ,\dots, \tilde{y}_{\nbin-1} \in \R_{>0}$ so that $(\tilde{y_1} ,\dots, \tilde{y}_{\nbin}) \notin \im(\Psi)$ and for all $y \in \im(\Psi)$ with $y_{\nbin} = \tilde{y}_{\nbin}$ holds that $y_1 >\tilde{y_1}   ,\dots, y_{\nbin-1}>\tilde{y}_{\nbin-1} $.
\end{lemma}
\begin{proof}
For $j = 1,\dots,\nspec-1$ let $\tilde{y}_j$ be given as
\begin{align*}
	\tilde{y}_j  &=  \int\limits_{I_{\nbin}} s_{j}(E) \exp\Big( -  E^{-c} \cdot \tilde{v}_{\nbin} \Big) \, \mathrm{d}E > 0,
\end{align*}
where $\tilde{v}_{\nbin} \in \R$ is the unique solution to
\begin{align*}
	\tilde{y}_{\nbin} &=  \int\limits_{I_{\nbin}} s_{\nbin}(E) \exp\Big( -  E^{-c} \cdot \tilde{v}_{\nbin} \Big) \, \mathrm{d}E.
\end{align*} 
From this follows also that for all  $y \in \im(\Psi)$ with $y_{\nbin} = \tilde{y}_{\nbin}$ and their corresponding solutions $v \in \R^{\nspec}$ it must hold that $v_{\nbin} = \tilde{v}_{\nbin}$. For $j = 1,\dots,\nspec-1$ we then have
\begin{align*}
	y_j &= \sum_{i=1}^{\nspec} \int\limits_{I_i} s_j(E) \exp\Big( -  E^{-c} \cdot v_i \Big) \, \mathrm{d}E \\
	&> \int\limits_{I_{\nbin}} s_j(E) \exp\Big( -  E^{-c} \cdot v_{\nbin} \Big) \, \mathrm{d}E \\
	&=  \int\limits_{I_{\nbin}} s_j(E) \exp\Big( -  E^{-c} \cdot \tilde{v}_{\nbin} \Big) \, \mathrm{d}E \\
	&= \tilde{y}_j.
\end{align*}
This implies directly that $(\tilde{y_1} ,\dots, \tilde{y}_{\nbin}) \notin \im(\Psi)$.
\end{proof}
In the same way one can show that if for some $j = 2,\dots,\nspec-1$ and $y_j$ we can find $\tilde{y_i}$ for $i = 1,\dots,j-1$ so that we have $y_1 >\tilde{y_1}   ,\dots, y_{j-1}>\tilde{y}_{j-1} $ for all  $y \in \im(\Psi)$ with $y_{j} = \tilde{y}_{j}$. Moreover, in the case that multiple entries of $y \in \im(\Psi)$ are known, the lower bounds for the rest of the entries can be given by the summation of the bounds from the individual entries.

Since $\nbin = \nspec$ will not hold in general, we investigate the case $\nbin > \nspec$ next. For this we can obtain uniqueness, if a subset of the sensitivity functions satisfies the requirements of theorem~\ref{theo:uniqueness_ideal}.
\begin{theorem}\label{theo:uniqueness_overdetermined}
Let $\nbin>\nspec$ and assume that for all $i \in \{1,\dots,\nspec\}$ we can find a sensitivity function $s_{b_i}$ so that $\int_{I_j}s_{b_i}(E)  \mathrm{d}E \neq 0$ for all $j\geq i$ and $\int_{I_j}s_{b_i}(E)  \mathrm{d}E = 0$ for all $j<i$. Then for each $y \in  \im(\Psi)$ there exists a unique $v \in \R^{\nspec}$ so that $\Psi (v) = y$.
\end{theorem}
\begin{proof}
Assuming that we have $v, \tilde{v} \in \R^{\nspec}$ with $\Psi(v) = \Psi(\tilde{v}) = y$, then for the function $\tilde{\Psi} : \R^{\nspec} \to \R_{>0}^{\nspec}$ with
\begin{align*}
	\tilde{\Psi}_j(v) &= \sum_{i=1}^{\nspec} \int\limits_{I_i} s_{b_j}(E) \exp\Big( -  E^{-c} \cdot v_i \Big) \, \mathrm{d}E
\end{align*}
for $j \in \{1,\dots,\nspec\}$ holds $\tilde{\Psi}(v) = \tilde{\Psi}(\tilde{v}) $. The same argumentation from theorem \ref{theo:uniqueness_ideal} gives us then that $v = \tilde{v}$.
\end{proof}

Lastly, in the case that $\nbin < \nspec$, or when there are energy intervals $I_i$ where there are no sensitivity functions satisfying the requirements from theorem \ref{theo:uniqueness_overdetermined}, we cannot find a unique solution $v \in \R^{\nspec}$ given a $y \in \im(\Psi) $ so that $\Psi(v) = y$. However we can adapt theorem \ref{theo:uniqueness_overdetermined} as follows.
\begin{theorem}\label{theo:uniqueness_underdetermined}
Let $\nbin < \nspec$ and $k \in \{1,\dots,\nspec\}$ be so that for all $i \in \{k,\dots,\nspec\}$ we can find a sensitivity function $s_{b_i}$ with $\int_{I_j}s_{b_i}(E)  \mathrm{d}E \neq 0$ for all $j\geq i$ and $\int_{I_j}s_{b_i}(E)  \mathrm{d}E = 0$ for all $j<i$, then for $v,\tilde{v} \in \R^{\nspec}$ follows from $\Psi(v) = \Psi(\tilde{v})$ that $v_i =  \tilde{v}_i $ for all $i \in \{k,\dots,\nspec\}$.
\end{theorem}
\begin{proof}
We define the function $\tilde{\Psi} : \R^{\nspec-k+1} \to \R^{\nspec-k+1}_{>0}$ with 
\begin{align*}
	\tilde{\Psi}_j(v) &= \sum_{i=1}^{\nspec-k+1} \int\limits_{I_{i+k-1}} s_{b_{j+k-1}}(E) \exp\Big( -  E^{-c} \cdot w_i \Big) \, \mathrm{d}E
\end{align*}
for $j \in \{1,\dots,\nspec-k+1\}$ and $w \in \R^{\nspec-k+1}$. For  $v, \tilde{v} \in \R^{\nspec}$ with $\Psi(v) = \Psi(\tilde{v}) = y$ we also have $\tilde{\Psi} ( v_k,\dots,v_{\nspec}) = \tilde{\Psi} ( \tilde{v}_k,\dots,\tilde{v}_{\nspec}) $ and therefore $v_i = \tilde{v}_i$ for all $i \in \{ k,\dots,\nspec \}$.
\end{proof}
The scheme for the requirements of the sensitivity functions can also be reversed to represent sensor bins created by using multiple X-ray sources as found in~\cite{Zhou_2024,Sandvold2024}. The bins would then satisfy $ \supp (s_b) \in [0, E ]$ for some bin specific energy value $E$. With this one could obtain the same results as in theorems~\ref{theo:uniqueness_ideal},~\ref{theo:uniqueness_overdetermined}, and~\ref{theo:uniqueness_underdetermined}, and lemma~\ref{lem:surjectivity}. 
With hybrid spectral CT systems as introduced in~\cite{Sandvold2024}, which both include several X-ray sources and sensors one could potentially be able to find combinations of X-ray tubes, sensors and materials, under which theorem~\ref{theo:diffeomorph_ideal} is applicable in a realistic scenario.

\subsection{Material-energy decomposition problem and result for the bin-sinogram decomposition}

In the three step approach we split the bin-sinogram decomposition into the bin decomposition and the material-energy decomposition. As the material-energy decomposition is a linear problem it is uniquely solvable if $A$ from \eqref{eq:A} is injective. In combination with theorem \ref{theo:uniqueness_overdetermined} this leads to the following result.
\begin{theorem}\label{theo:bin_sinogram_decomposition_ideal}
Let  $\ker(A) = \{0\}$ and for all $i \in \{1,\dots,\nspec\}$ let there exist a sensitivity function $s_{b_i}$ so that $\int_{I_j}s_{b_i}(E)  \mathrm{d}E \neq 0$ for all $j\geq i$ and $\int_{I_j}s_{b_i}(E)  \mathrm{d}E = 0$ for all $j<i$. Then $\Phi$ is injective.
\end{theorem}
\begin{proof}
Theorem \ref{theo:uniqueness_overdetermined} gives us that $\Psi$ is injective. Since $\ker(A) = \{0\}$ we also know that $A$ is injective. Therefore $\Psi\circ A$ is also injective. Because of \eqref{eq:Phi_Psi_A} this then holds also for $\Phi$.
\end{proof}
If the sensitivity functions do not meet the requirements from theorems \ref{theo:uniqueness_overdetermined} and \ref{theo:bin_sinogram_decomposition_ideal} we can apply theorem \ref{theo:uniqueness_underdetermined}.	For this we define the reduced form of $A$ as $A^{(k)}  \in \R^{\nspec -k+1\times \nmat} $ with
\begin{align}
\label{eq:A_reduced}
A^{(k)} = \begin{pmatrix}
	a_1^{k} & \cdots & a_{\nmat}^{k} \\ \vdots & & \vdots \\ a_1^{\nspec} & \cdots & a_{\nmat}^{\nspec}
\end{pmatrix}.
\end{align}
With this we can state:
\begin{theorem}\label{theo:bin_sinogram_decomposition_non_ideal}
In case there exists a $k \in \{1,\dots,\nspec\}$ so that $\ker(A^{(k)}) = \{0\}$ and for all $i \in \{k,\dots,\nspec\}$ we can find a sensitivity function $s_{b_i}$ with $\int_{I_j}s_{b_i}(E)  \mathrm{d}E \neq 0$ for all $j\geq i$ and $\int_{I_j}s_{b_i}(E)  \mathrm{d}E = 0$ for all $j<i$, it holds for all  $y \in \im(\Phi)$ that there only exists one $z \in \R^{\nmat}$ so that $\Phi(z) = y$.
\end{theorem}
\begin{proof}
For a given $y \in \im(\Phi)$ let $z, \tilde{z} \in \R^{\nmat}$ satisfy $y = \Phi(z) = \Phi(\tilde{z})$. For $v = Az$ and $\tilde{v} = A \tilde{z}$ then holds $y = \Psi(v) = \Psi(\tilde{v})$. Thanks to theorem \ref{theo:uniqueness_underdetermined} we know that $v_i = \tilde{v}_i$ for all $i \in \{ k,\dots,\nspec\}$ and therefore that $A^{(k)} z = A^{(k)} \tilde{z}$. Since $A^{(k)}$ is injective we must have $z =  \tilde{z}$.
\end{proof}

\section{Numerical example}
\label{sec:Num_Ex}        
To demonstrate the applicability of the three step approach, we first describe the procedure on a small two detector bin two material (iodine and gadolinium) decomposition problem shown in figure~\ref{fig:ThreeStepStart}. After that we do further numerical experiments in which we show how ignoring Compton scattering effects the results, which impact the choice of $c$ makes, and the reconstructions under the scenarios that the support of the bins does not satisfy the uniqueness criterion. For all experiments we used the data provided in~\cite{Mory_2018,Mory_2019_Git}.
\begin{figure}[t]
\centering
\includegraphics[width=\linewidth]{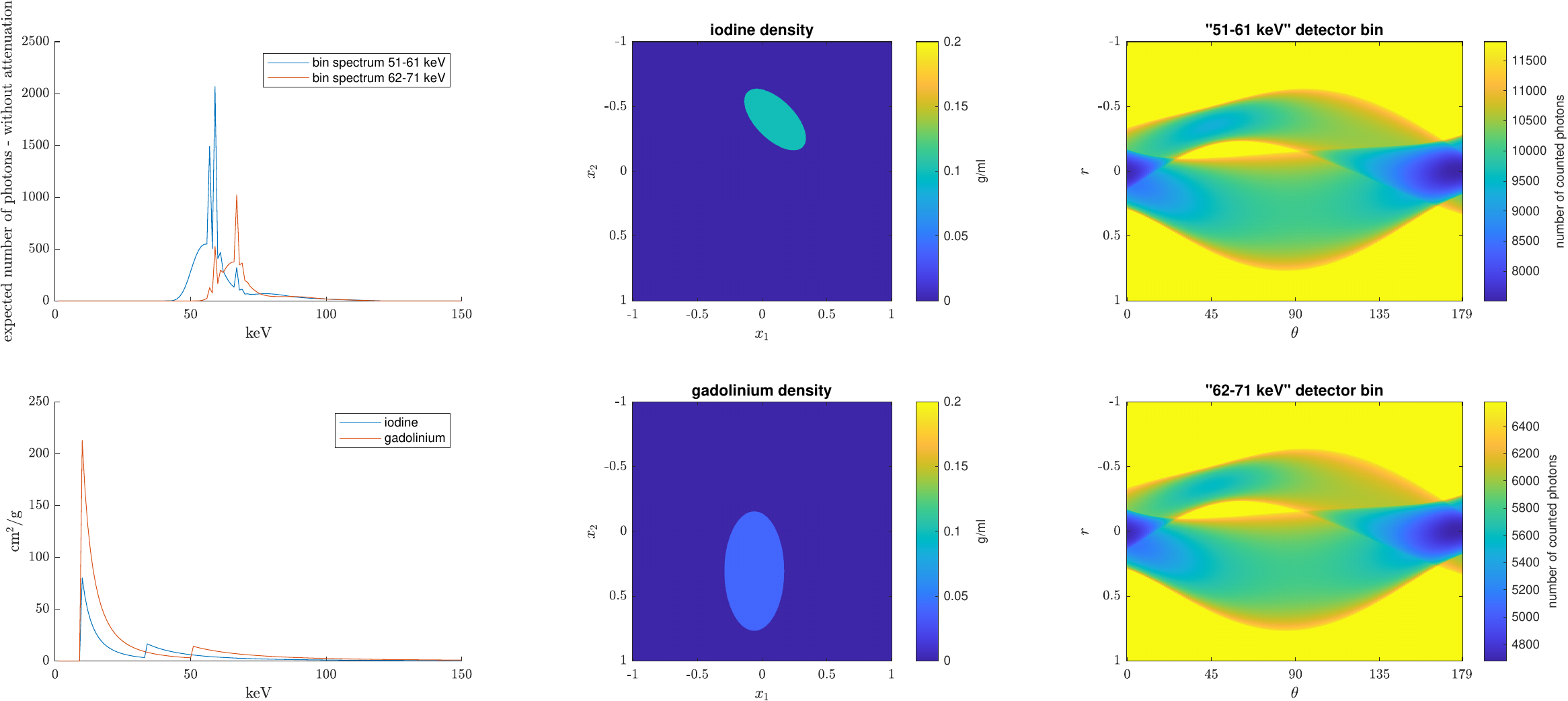}
\caption{Data for the material decomposition problem, where the detector response spectra and the material attenuation spectra are out of the provided data from~\cite{Mory_2018,Mory_2019_Git}.}
\label{fig:ThreeStepStart}
\end{figure}

From the detector bins and material attenuation spectra follows that $\nbin = 2$ and $\nspec = 3$ with $s_1$ represented by the "$51-61\text{keV}$" detector bin and $s_2$ by  "$62-71\text{keV}$". The intervals given by the binding energies of the electrons are $I_1 = [10 \text{keV}, 34 \text{keV})$, $I_2 = [34 \text{keV}, 51 \text{keV})$ and  $I_3 = [51 \text{keV}, 150 \text{keV})$. For this the detector bin "$51-61\text{keV}$"  satisfies $\int_{I_1} s_1(E) \mathrm{d}E = 0$, $\int_{I_2} s_1(E) \mathrm{d}E \neq 0$ and $\int_{I_3} s_1(E) \mathrm{d}E \neq 0$. For the second bin "$62-71\text{keV}$" we have $\int_{I_1} s_2(E) \mathrm{d}E = 0$, $\int_{I_2} s_2(E) \mathrm{d}E = 0$ and $\int_{I_3} s_2(E) \mathrm{d}E \neq 0$. Theorem \ref{theo:uniqueness_underdetermined} gives us therefore that we can reconstruct the second and third energy interval. A value of $c=2.6$ best fits our data. The matrix $A$ is then given by
\begin{align*}
A = \begin{pmatrix}
	0.8469 & 0.3193 \\ 0.8469 & 1.5903 \\ 3.8967 & 1.5903
\end{pmatrix} \cdot 10^{5},
\end{align*}
and the reduced matrix $A^{(2)}$
as
\begin{align*}
A^{(2)} = \begin{pmatrix}
	0.8469 & 1.5903 \\ 3.8967 & 1.5903
\end{pmatrix} \cdot 10^{5}.
\end{align*}
Since $\ker\left(A^{(2)}\right) = \{0\}$ theorem \ref{theo:bin_sinogram_decomposition_non_ideal} gives us that the bin-sinogram decomposition problem is injective.

\begin{figure}
\centering
\includegraphics[width=\linewidth]{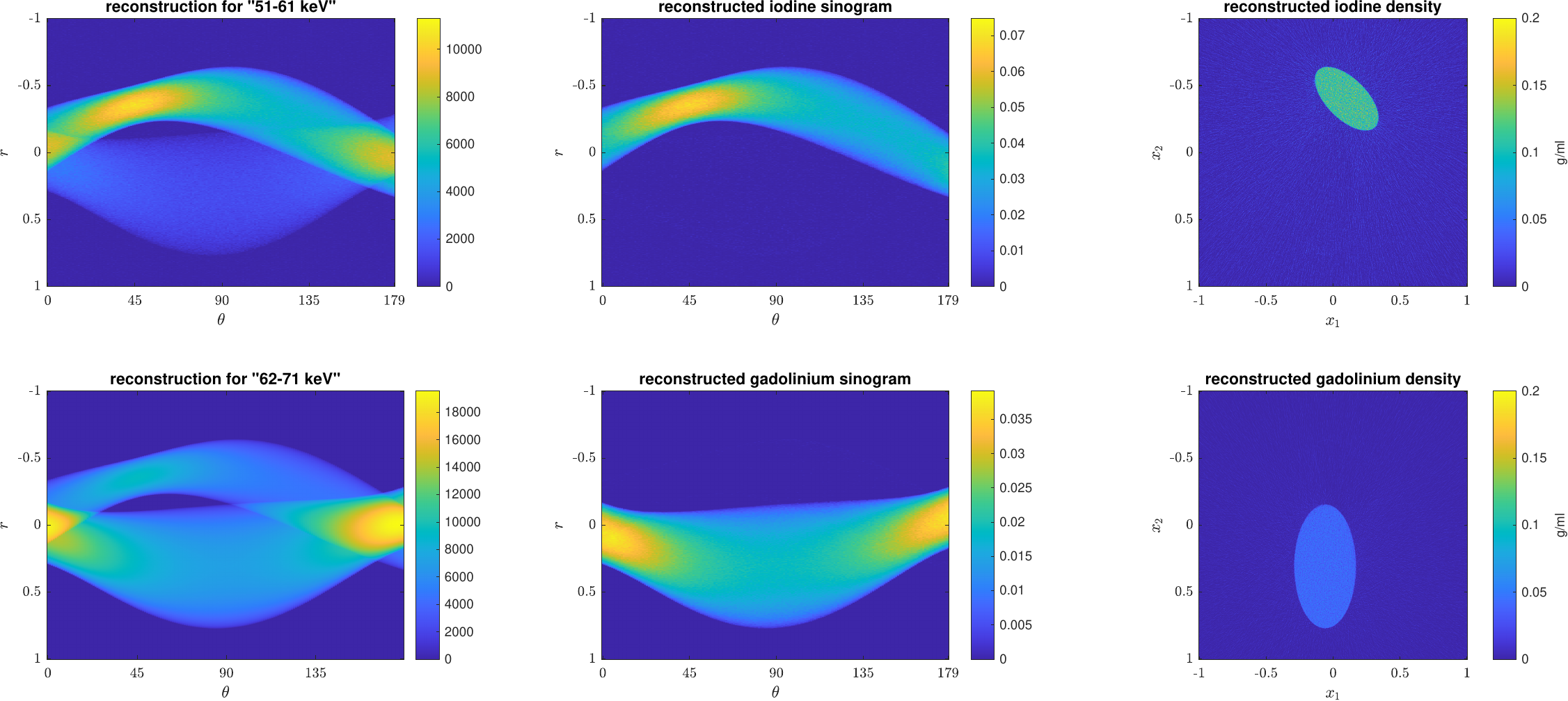}
\caption{The three reconstruction steps of approach~\ref{Approche1} using algorithm~\ref{alg:RecursiveBinDecomposition} for the material decomposition problem depicted in figure~\ref{fig:ThreeStepStart}.}
\label{fig:ThreeStepRec}
\end{figure} 
\begin{algorithm}
\caption{(Recursive bin-decomposition)}
\label{alg:RecursiveBinDecomposition}
\begin{algorithmic}[1]
	\Require Signal $Y \in \R^{\nbin}$ with the according $s_{b_i}$ for all $i = k,\dots,\nspec$ as in theorem \ref{theo:uniqueness_underdetermined}, reduced matrix $A^{(k)}$
	\medskip
	\For{$j = \nspec , \dots , k$}
	\State Solve $Y_{b_j}(\ell) =  \int_{I_j} s_{b_j}(E) \exp ( -E^{-c} \cdot v_j)$ for $v_j$ 
	\State Set $Y_{b_i}(\ell) = Y_{b_i}(\ell) -\int_{I_j} s_{b_i}(E) \exp ( -E^{-c} \cdot v_j) $ for all $i = k,\dots,j-1$
	\EndFor
	\Ensure $v_k,\dots,v_{\nspec}$
\end{algorithmic}
\end{algorithm}

For the reconstruction we used approach~\ref{Approche1}, where we reconstructed the bin decomposition problem $\Psi(v(\theta,r)) = Y(\theta,r)$ for each point $(\theta,r)$ via the recursive method described in algorithm~\ref{alg:RecursiveBinDecomposition}. In this we solve the entries of $v(\theta,r)$ from last to first and subtract the results from the still unsolved detector readings.	
The requirements on the sensitivity functions from theorem~\ref{theo:bin_sinogram_decomposition_non_ideal} guarantee us that each step is a one dimensional convex problem, which we solve in our case with Newton's method.
The material-energy decomposition $Az(\theta,r) = v(\theta,r)$ can again for each pair $(\theta,r)$ be solved via the least-squares method.
For the Radon inversion we use the filtered back projection with cosine filter.	
For all reconstructions we utilized $Y_b \in \R^{180 \times 1024}$ with added Gaussian white noise. And to create the data we employed the discrete versions of  $s_b$ and $\mu_m$, which consist of $150$ discrete values from $1 \text{keV}$ to $150 \text{keV}$.
The results and the individual steps are depicted in figure~\ref{fig:ThreeStepRec}.

\begin{figure}[t]
	\centering
	\includegraphics[width=\linewidth]{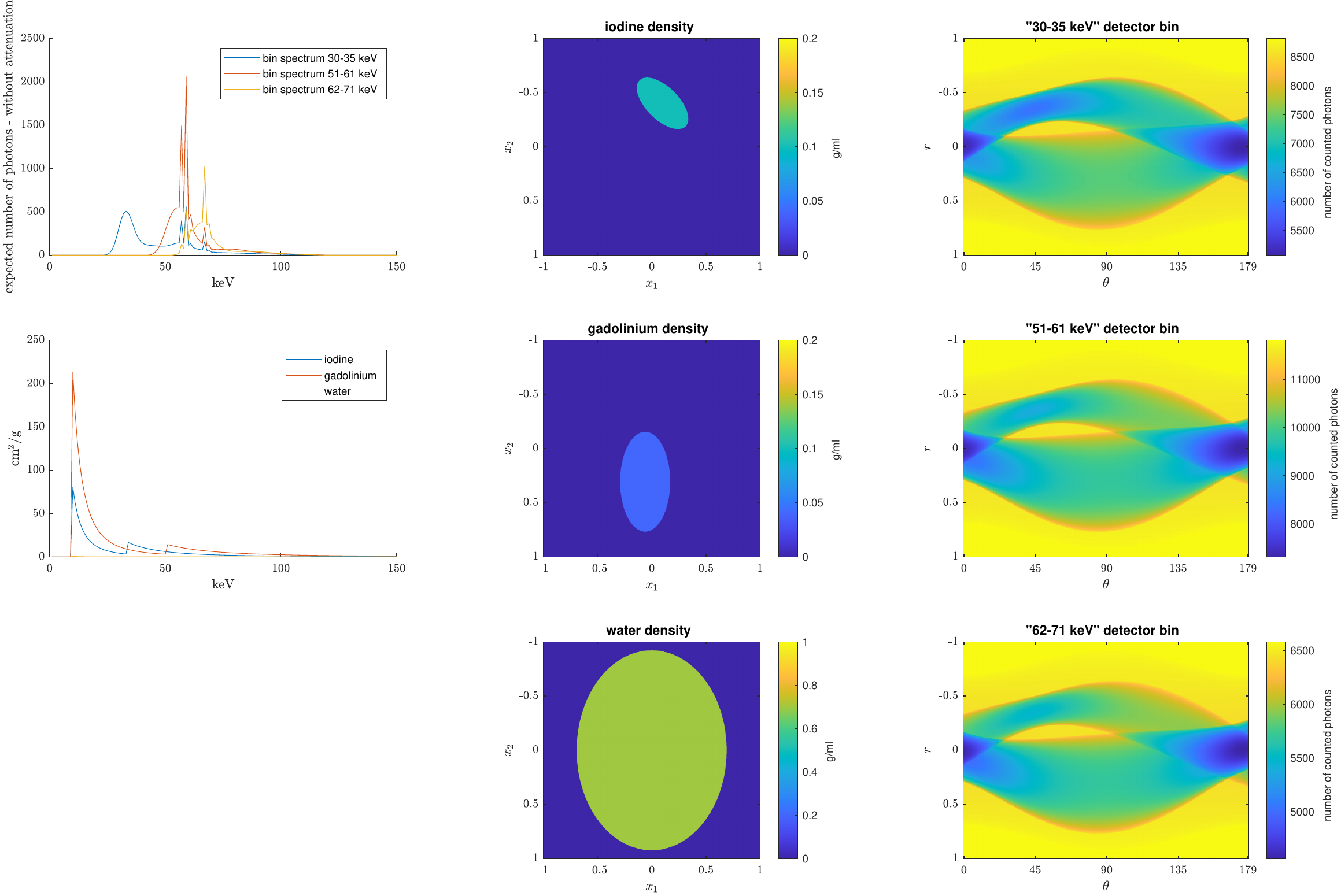}
	\caption{Data for the material decomposition problem with three material, where the detector response spectra and the material attenuation spectra are out of the provided data from~\cite{Mory_2018,Mory_2019_Git}.}
	\label{fig:ThreeStepStart3Mat}
\end{figure}

\begin{figure}[t]
	\centering
	\includegraphics[width=\linewidth]{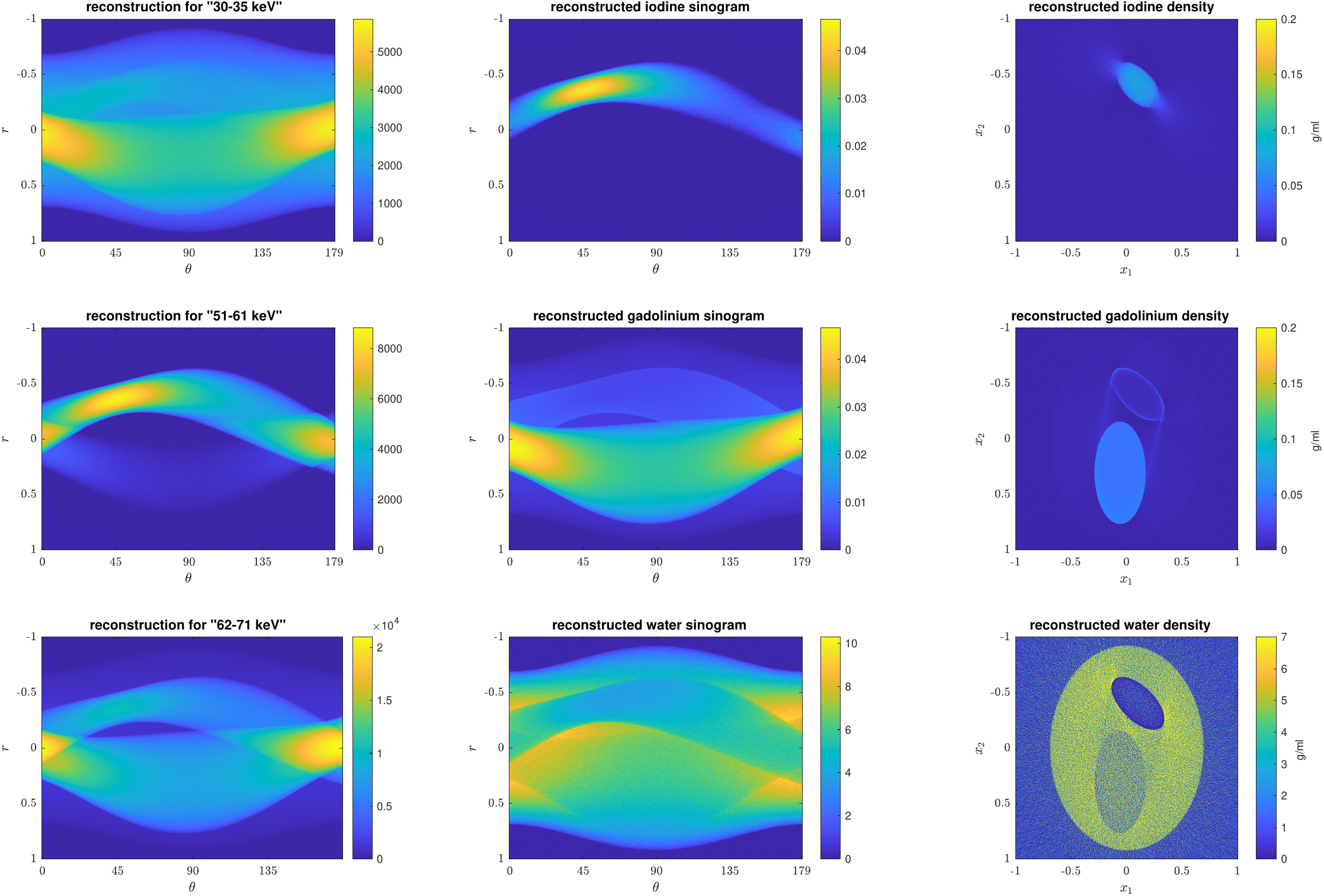}
	\caption{The three reconstruction steps of approach~\ref{Approche1} using algorithm~\ref{alg:RecursiveBinDecomposition} with $c = 2.6$ for the material decomposition problem depicted in figure~\ref{fig:ThreeStepStart}.}
	\label{fig:ThreeStepRec3Mat}
\end{figure}

We will now incorporate a larger sphere of water to the material decomposition problem, as Compton scattering has a larger effect on the absorption properties of water. The new data is depicted in figure~\ref{fig:ThreeStepStart3Mat} with the results in figure~\ref{fig:ThreeStepRec3Mat}. As the absorption of water is not following the absorption curve given by $A$, the material decomposition fails especially for the density of water and the differentiation from iodine and the quantitative value of water.

\begin{figure}[t]
	\begin{enumerate}[label=(\alph*)]
		\item\label{c=2.6Plots}
		\includegraphics[width=0.925\textwidth]{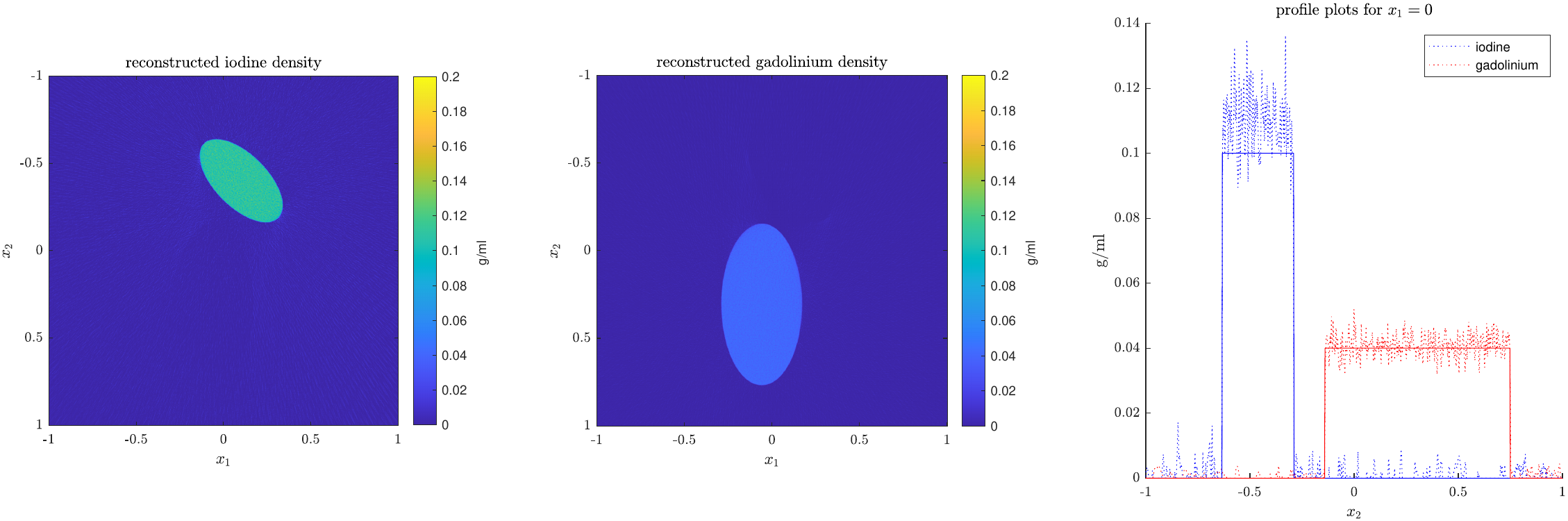}
		\bigskip
		\item\label{c=3.0Plots}
		\includegraphics[width=0.925\textwidth]{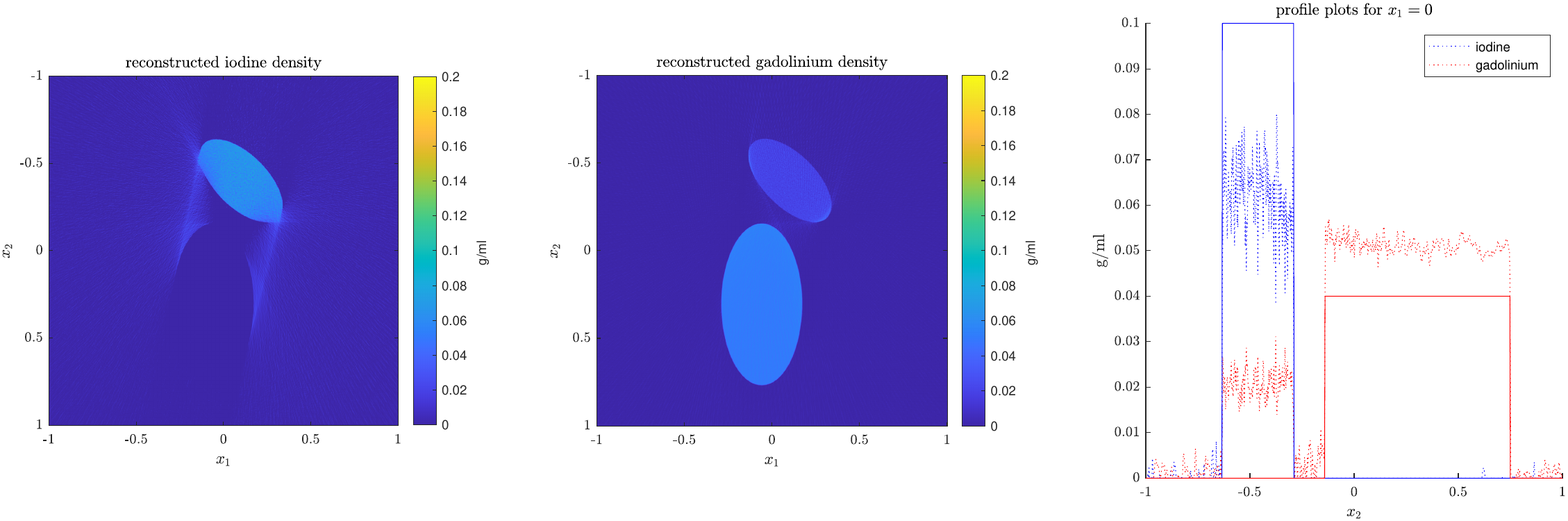}
		\bigskip
		\item\label{c=3.5Plots}
		\includegraphics[width=0.925\textwidth]{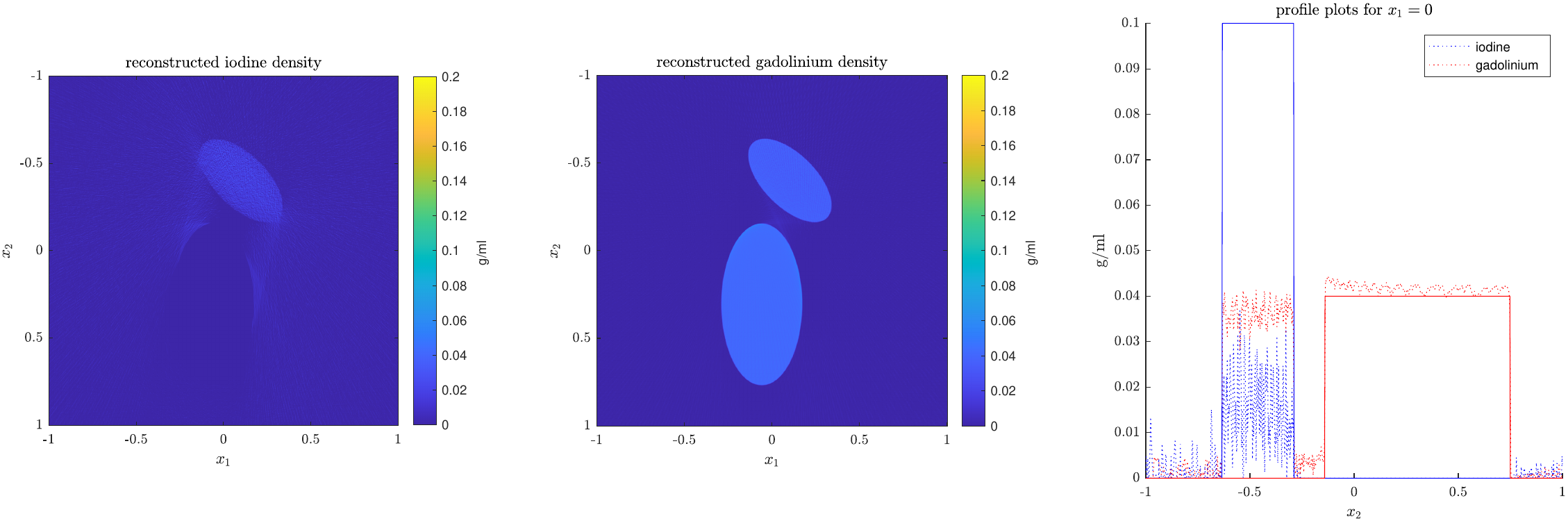}
	\end{enumerate}
	\caption{The three reconstructions of iodine and gadolinium using  \ref{c=2.6Plots} $c = 2.6 $, \ref{c=3.0Plots} $c = 3 $ and \ref{c=3.5Plots} $c = 3.5 $, where the solid line in the profile plot is the ground truth and the dotted line is the reconstruction.}
	\label{fig:ThreeStepForC}
\end{figure}

Taking again a closer look at the two material problem out of figure~\ref{fig:ThreeStepStart}, we want to investigate the impact the selection of $c$ makes. As seen in figure~\ref{fig:REC} choosing a different value for $c$ changes how accurately we can recreate the individual materials. On one side it effects bin-decomposition especially when it is done in a recursive way such as in algorithm~\ref{alg:RecursiveBinDecomposition}, but also by controlling $A$ and therefore our ability to identify the materials. When using the same material approximation as in  figure~\ref{fig:REC}, one can see in figure~\ref{fig:ThreeStepForC}\ref{c=2.6Plots} that $c = 2.6$ gives the best results and  with $c =3$ a  part of iodine is misidentified as gadolinium. \ref{c=3.0Plots} also shows that the amount of gadolinium is overestimated. Setting $c = 3.5$ one can see in \ref{c=3.5Plots} that the most of iodine is mislabeled, and the reconstructed amount of gadolinium is even larger.

\begin{figure}[t]
	\begin{enumerate}[label=(\alph*)]
		\item\label{c=2.6PlotsNewA}
		\includegraphics[width=0.925\textwidth]{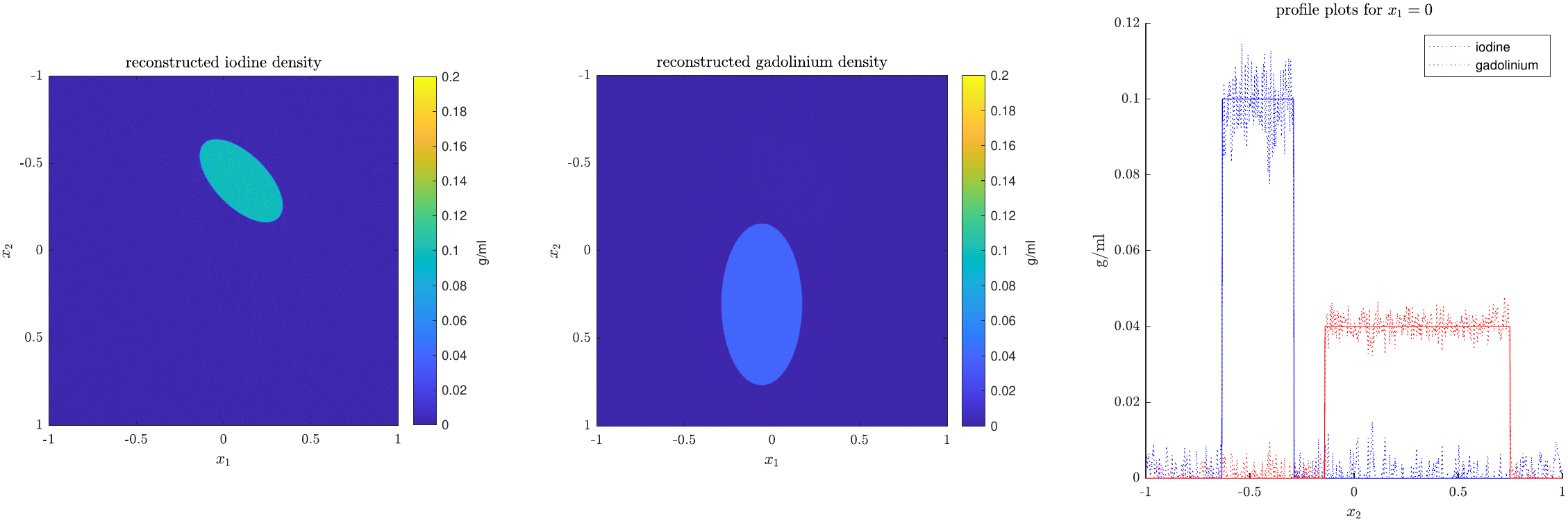}
		\bigskip
		\item\label{c=3.0PlotsNewA}
		\includegraphics[width=0.925\textwidth]{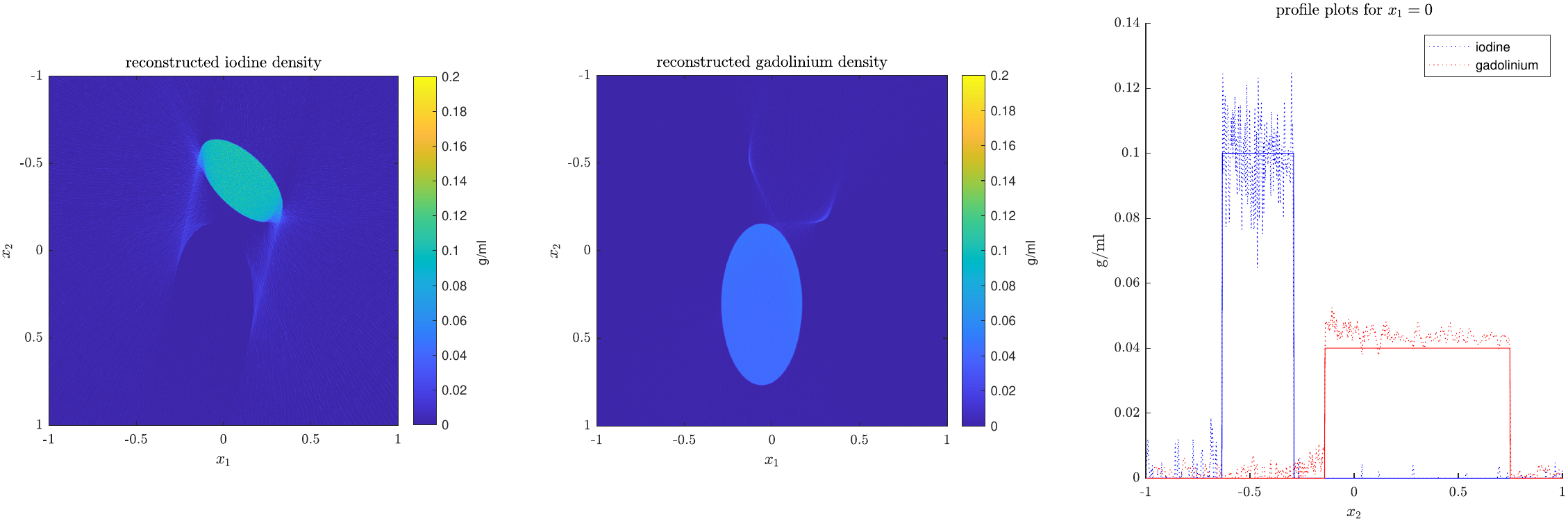}
		\bigskip
		\item\label{c=3.5PlotsNewA}
		\includegraphics[width=0.925\textwidth]{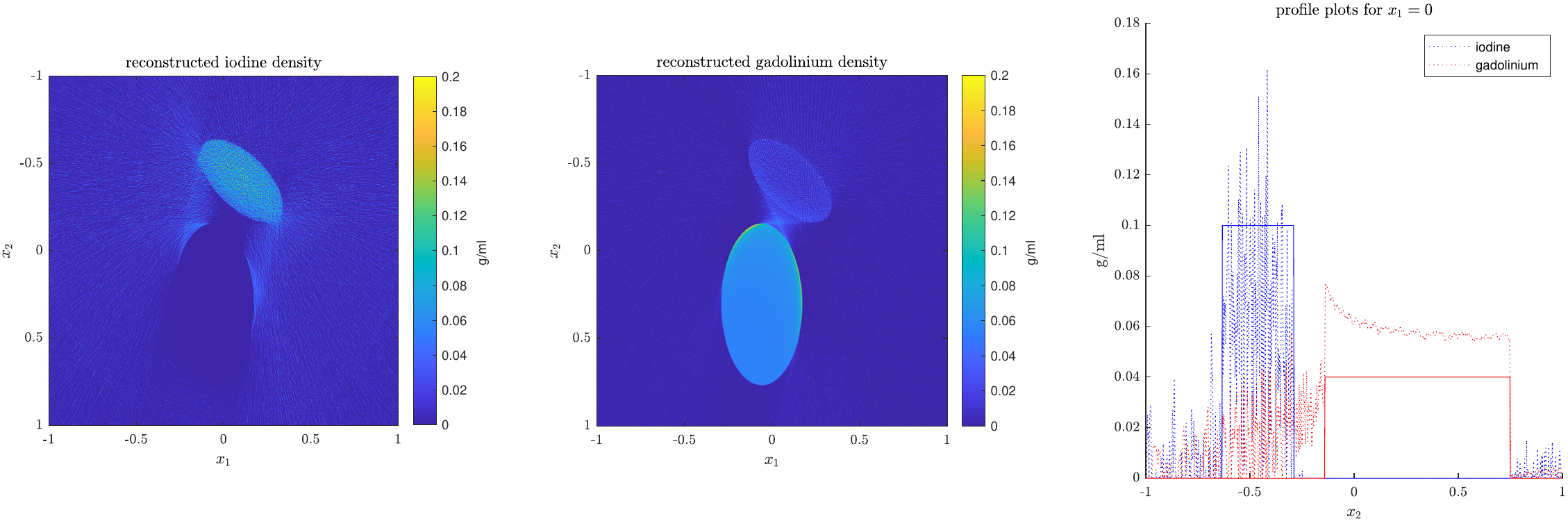}
	\end{enumerate}
	\caption{The three reconstructions of iodine and gadolinium using  \ref{c=2.6PlotsNewA} $c = 2.6 $, \ref{c=3.0PlotsNewA} $c = 3 $ and \ref{c=3.5PlotsNewA} $c = 3.5 $ with $A$ calculated by \eqref{eq:calcA}, where the solid line in the profile plot is the ground truth and the dotted line is the reconstruction.}
	\label{fig:ThreeStepForCNewA}
\end{figure}
As the value of $c$ effects the material-energy decomposition, we want to investigate how choosing a different $A$ can improve the reconstruction. In the previous experiments we calculated $A$ via
\begin{align*}
	a^{(i)}_m = \frac{\mu_m(E)}{E^{-c}},
\end{align*}
where $E$ is the applicable electron binding energy. We now choose 
\begin{align}\label{eq:calcA}
	a^{(i)}_m = \left(\Psi^{-1}(\Phi(z))\right)_i
\end{align}
for $z_m = 1$ for all $\tilde{m} \neq m$ and $z_{\tilde{m}} = 0$. This improves the material decomposition significantly. In the case of $c=2.6$ the quantitative values are even closer to the ground truth, as seen in figure~\ref{fig:ThreeStepForCNewA}\ref{c=2.6PlotsNewA}. For $c =3.0$ we can also see a better separation of materials, when comparing figure~\ref{fig:ThreeStepForCNewA}\ref{c=3.0PlotsNewA} with figure~\ref{fig:ThreeStepForC}\ref{c=3.0Plots}. Figure~\ref{fig:ThreeStepForCNewA}\ref{c=3.5PlotsNewA} shows that for $c = 3.5$ the materials can still be separated, however the error from the bin decomposition in the first energy bin leads to an overall larger error.

\begin{figure}
	\begin{enumerate}[label=(\alph*)]
		\item\label{Shift0}
		\includegraphics[width=0.925\textwidth]{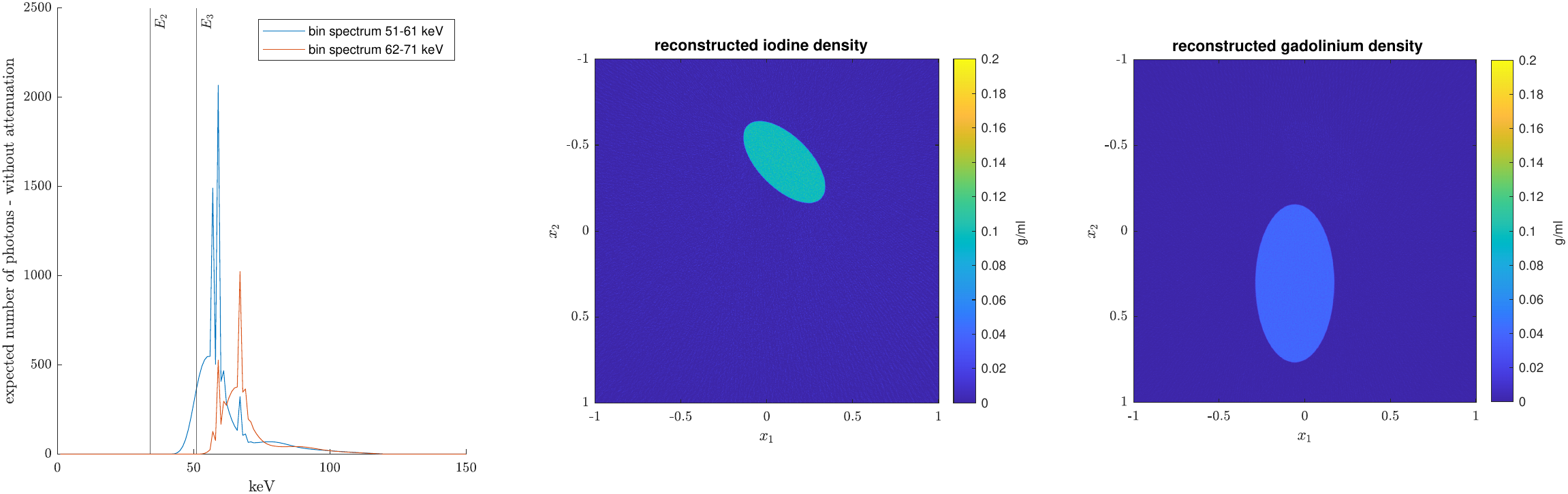}
		\smallskip
		\item\label{Shift5}
		\includegraphics[width=0.925\textwidth]{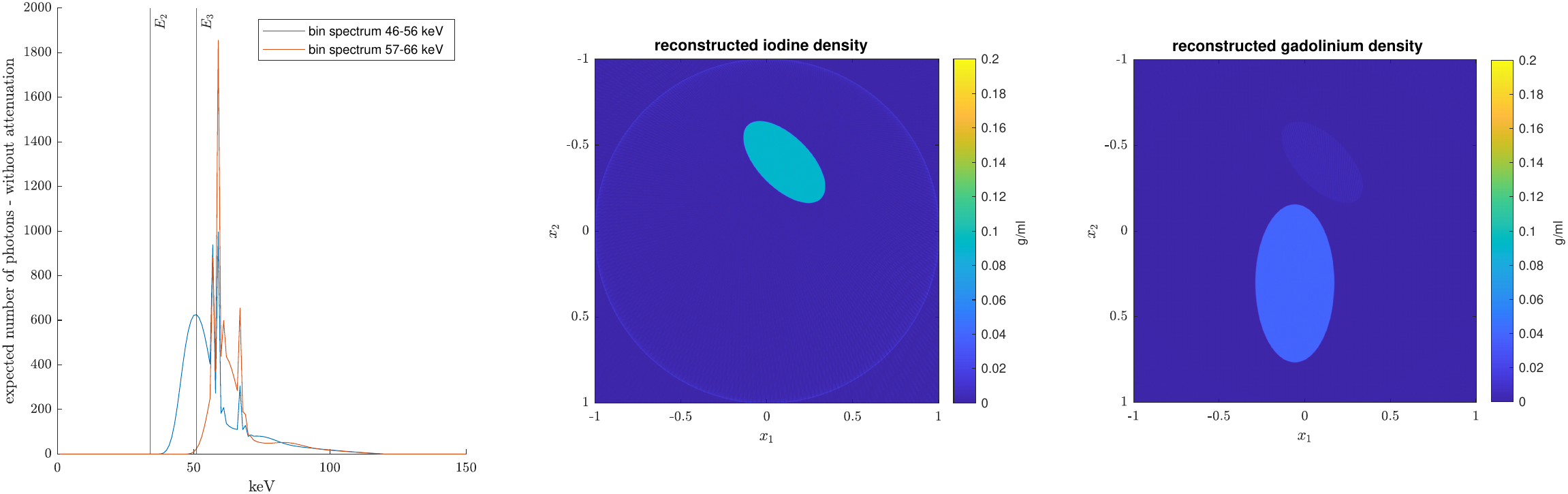}
		\smallskip
		\item\label{Shift10}
		\includegraphics[width=0.925\textwidth]{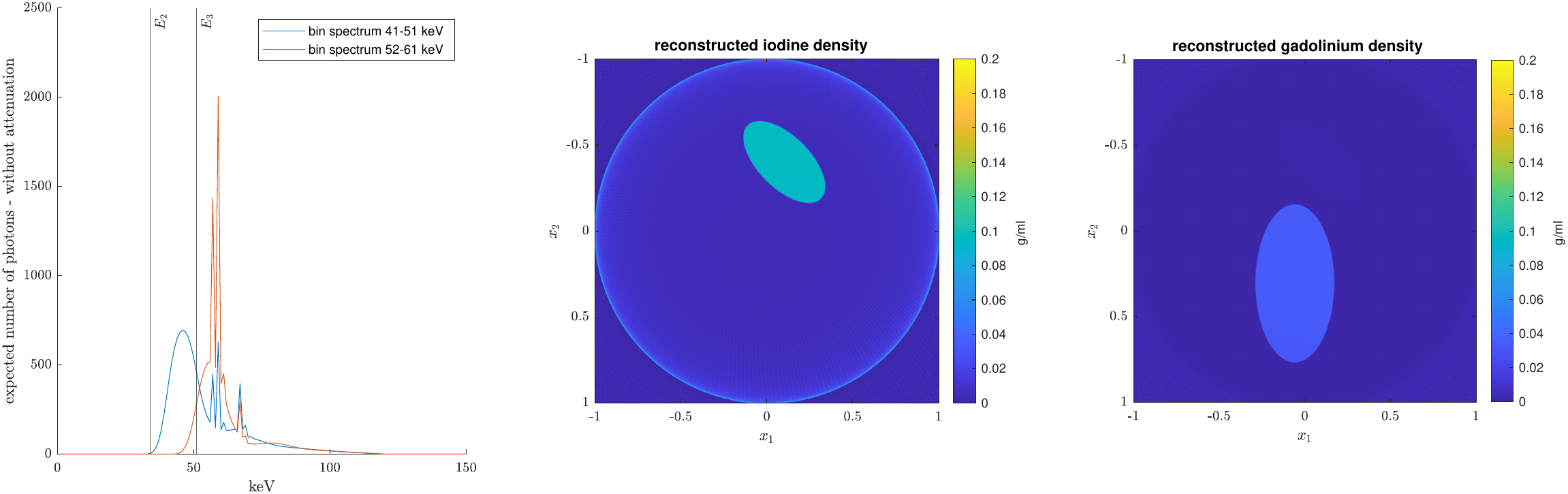}
		\smallskip
		\item\label{Shift15}
		\includegraphics[width=0.925\textwidth]{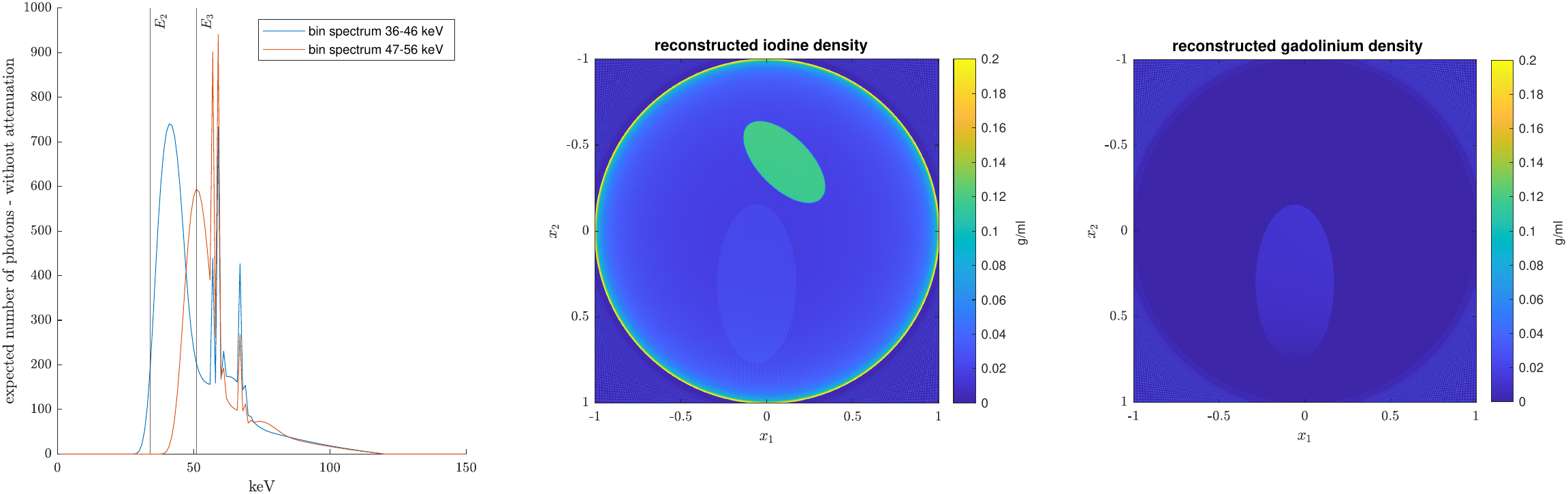}
	\end{enumerate}
	\caption{The reconstruction of iodine and gadolinium using c = $2.6$, for sensitivity functions following the uniqueness condition~\ref{Shift0}, the second bin exceeding interval $I_3$ slightly~\ref{Shift5}, exceeding interval $I_3$ significantly~\ref{Shift10} and the support of both sensitivity functions exceeding $I_2 \cup I_3$ and respectively $I_3$~\ref{Shift15}.}
	\label{fig:ThreeStepForShiftedBinSupp}
\end{figure}

We now investigate the necessities of the uniqueness condition. When $\ker(A^{(k)}) \neq \{0\}$, there is at least one material which can be represented by a linear combination of the other materials, which implies that the material-energy decomposition does not omit a unique solution. We will therefore only consider the case where the bin sensitivity functions are also non zero  outside the required energy intervals. We will again use the scenario described in figure~\ref{fig:ThreeStepStart}. In the cases depicted in \ref{fig:ThreeStepForShiftedBinSupp}\ref{Shift5}, \ref{Shift10} and \ref{Shift15} we have $\int_{I_3} s_2(E) \mathrm{d}E \neq 0$. Therefore theorem~\ref{theo:bin_sinogram_decomposition_non_ideal} is not satisfied. In figure~\ref{fig:ThreeStepForShiftedBinSupp}\ref{Shift5} we still obtain a satisfactory result, as $\int_{I_3} s_2(E) \mathrm{d}E \gg \int_{I_2} s_2(E) \mathrm{d}E $. For the scenarios \ref{Shift10} and \ref{Shift15}, the reconstruction loses accuracy. Especially in the last case where we also have $\int_{I_1} s_1(E) \mathrm{d}E \neq 0$.

\section{Conclusion}
\label{sec:conclusion}        
We used similarities in the material attenuation spectra based on the photoelectric effect, which is the main contributor of photon absorption in case of lower energy X-rays and materials with high atomic number, 
to further separate the non linear part of the  MSCT problem. From this followed the three step representations~\eqref{eq:Y_Phi_Psi_A} and~\eqref{eq:Y_Psi_R_A}  with the corresponding reconstruction approaches~\ref{Approche1} and~\ref{Approche2}. 
In the former the material-energy decomposition is done directly after the bin-decomposition and in the latter it is done in image space. In both cases the newly introduced bin-decomposition aims to recover individual physical effects contributing to the bin measurements. The material-energy decomposition then tries to identify the individual materials based on their absorption properties given by  the matrix $A$.
For the bin decomposition problem we proved conditions for well-posedness with theorem~\ref{theo:diffeomorph_ideal} and uniqueness with theorem~\ref{theo:uniqueness_ideal} and theorem~\ref{theo:uniqueness_overdetermined}.
The material-energy decomposition can then be evaluated as a linear problem by looking at the properties of $A$ to determine, if the materials can be isolated.
Together this gives uniqueness conditions for the bin-sinogram decomposition problem in form of theorem~\ref{theo:bin_sinogram_decomposition_ideal} and theorem~\ref{theo:bin_sinogram_decomposition_non_ideal}. 
All these conditions only depend on the relation of sensitivity functions to the binding energies of the individual materials and the values of the corresponding factors $a_m^{(i)}$ used to represent the material attenuation spectra in~\eqref{eq:material-attenuation-spectrum}.

Using simulated MSCT hardware we showed that our conditions for uniqueness are applicable and with the implementation of the reconstruction approach~\ref{Approche1} via algorithm~\ref{alg:RecursiveBinDecomposition} we also reconstructed a two material two detector-bin MSCT problem with satisfactory results.
Our numerical results showed further, how $c$, the choice of sensitivity bins, and Compton scattering affect the reconstruction.

The three step approach has also multiple future paths, which are open to further development. One is implementing other physical properties in the modelling of the material attenuation function, to obtain an even better reconstruction of the material properties.
This could also give a better applicability of the result in the scenario of higher X-ray energies and materials with small atomic number.
Another direction is the implementation of a non recursive reconstruction of bin-reading function. In this case the reconstruction errors for one energy value would not multiply in the following values. This direction also has the possibility of incorporating sensitivity functions not following the requirements from theorem~\ref{theo:bin_sinogram_decomposition_non_ideal}.

\bmhead{Funding}
This work was supported by Deutsche Forschungsgemeinschaft (DFG) - Project number 578955408
\bmhead{Acknowledgement}
This version of the article has been accepted for publication, after peer review but is not the Version of Record and does not reflect post-acceptance improvements, or any corrections. The Version of Record is available online at: \url{https://www.doi.org/10.1007/s11220-026-00873-w}.
\bibliography{references}

\end{document}